\documentclass[11pt]{article}
\usepackage[latin1]{inputenc}

\usepackage[a4paper,scale={.8,.8}]{geometry}
\usepackage[usenames, dvipsnames]{xcolor}
\usepackage[colorlinks=true, pdfstartview=FitV,linkcolor=ForestGreen,citecolor=ForestGreen, urlcolor=blue]{hyperref}

\makeatletter
\renewcommand\subparagraph{\@startsection {subparagraph}{5}{\z@ }{3.25ex \@plus 1ex
 \@minus .2ex}{-1em}{\normalfont \normalsize \bfseries }}%
\makeatother

\usepackage{amsmath,amssymb,amsfonts,abraces,relsize}
\usepackage{fourier}
\usepackage{caption,graphicx}
\usepackage{esvect}

\DeclareFontFamily{U}{MnSymbolC}{}
\DeclareSymbolFont{MnSyC}{U}{MnSymbolC}{m}{n}
\DeclareFontShape{U}{MnSymbolC}{m}{n}{
    <-6>  MnSymbolC5
   <6-7>  MnSymbolC6
   <7-8>  MnSymbolC7
   <8-9>  MnSymbolC8
   <9-10> MnSymbolC9
  <10-12> MnSymbolC10
  <12->   MnSymbolC12}{}
\DeclareMathSymbol{\intprod}{\mathbin}{MnSyC}{'270}

\newcommand{\N}{\mathbb{N}}
\newcommand{\Z}{\mathbb{Z}} 
\newcommand{\R}{\mathbb{R}}
\newcommand{\C}{\mathbb{C}} 
\newcommand{\T}{\mathbb{T}} 
\renewcommand{\div}{\operatorname{div}}

\newcommand{\eg}[1][~]{\textit{e.g.#1}}
\newcommand{\ie}[1][~]{\textit{i.e.#1}}

\newtheorem{thm}{Theorem}
\newtheorem{cor}[thm]{Corollary}
\newtheorem{prop}[thm]{Proposition}
\newtheorem{lemma}[thm]{Lemma}
\newtheorem{dfn}[thm]{Definition}
\newcommand{\proof}{\par\noindent\textit{Proof.} }
\newcommand{\cqfd}{\hfill\rule{1ex}{1ex}}
\newtheorem{ex}[thm]{Example}
\newtheorem{claim}[thm]{Claim}

\newtheorem{preremark}[thm]{Remark}
\newenvironment{rmk}%
  {\begin{preremark}\upshape}{\hfill{\small $\square$}\end{preremark}}
  
\newcommand{\norme}[2][]{\| #2 \|_{#1}}

\def\p{\partial}

\def\io{{\infty}}

\def\curl{\operatorname{curl}}

\def\range{\operatorname{ran}}

\def\Id{\operatorname{Id}}

\def\dive{\operatorname{div}}

\def\N{\mathbb N}
\def\Z{\mathbb Z}

\def\R{\mathbb R}
\def\C{\mathbb C}
\DeclareMathOperator{\rot}{\mathbf{C}}
\def\poscal#1#2{\langle#1,#2\rangle}

\def\norm#1{\Vert#1\Vert}

\def\val#1{\vert#1\vert}

\def\l2{L^2(\R^{n})}
\def\L2{L^2(\R^{2n})}

\def\tr#1{{^t}#1}

\let \dis=\displaystyle

\let \dis=\displaystyle

\def\mat22#1#2#3#4{\begin{pmatrix}#1&#2\\ #3&#4\end{pmatrix}}

\def\mattre#1#2#3{\begin{pmatrix}#1\\ #2\\#3\end{pmatrix}}

\def\XXint#1#2#3{{\setbox0=\hbox{$#1{#2#3}{\int}$}
     \vcenter{\hbox{$#2#3$}}\kern-.5\wd0}}

\def\beq{\begin{equation}}
\def\eeq{\end{equation}}

\def\tr#1{{^t}\!\!#1}

\title{On some properties of the curl operator\\and their consequences for the Navier-Stokes system}
\author{
Nicolas Lerner\\[1ex]
\small Sorbonne Universit\'e (formerly Paris VI)\\
\small Institut de Mathématiques de Jussieu, UMR 7586\\
\small Campus Pierre et Marie Curie, 4 Place Jussieu\\
\small 75252 Paris cedex 05\\
\small\texttt{nicolas.lerner@imj-prg.fr}
\and
Fran\c{c}ois Vigneron\\[1ex]
\small Université de Reims Champagne-Ardenne\\
\small Laboratoire de Mathématiques de Reims, UMR 9008\\
\small Moulin de la Housse, BP 1039\\
\small 51687 Reims cedex 2\\
\small\texttt{francois.vigneron@univ-reims.fr}
}
\date{December 25, 2021}
\begin{document}
\maketitle

\begin{center}\small\sl
In honor of our friend Professor Chaojiang Xu, on the occasion of his 65th birthday.
\end{center}

\begin{abstract}
We investigate some geometric properties of the $\curl$ operator, based on its diagonalization
and its expression as a non-local symmetry of the pseudo-derivative $(-\Delta)^{1/2}$ among divergence-free vector fields
with finite energy. In this context, we introduce the notion of spin-definite fields, \ie eigenvectors of $(-\Delta)^{-1/2}\curl$.
The two spin-definite components of a general 3D incompressible flow untangle the right-handed motion from the left-handed one. 

Having observed that the non-linearity of Navier-Stokes has the structure of a cross-product
and its weak (distributional) form is a determinant that involves the vorticity, the velocity and a test function,
we revisit the conservation of energy and the balance of helicity in a geometrical fashion. We show that in the case
of a finite-time blow-up, both spin-definite components of the flow will explose simultaneously and with equal rates,
\ie singularities in 3D are the result of a conflict of spin, which is impossible in the poorer geometry of 2D flows.
We investigate the role of the local and non-local determinants $$\int_0^T\int_{\R^3}\det(\curl u, u, (-\Delta)^{\theta} u)$$
and their spin-definite counterparts, which drive the enstrophy and, more generally, are responsible for
the regularity of the flow and the emergence of singularities or quasi-singularities.
As such, they are at the core of turbulence phenomena.\\[1ex]
\textbf{MSC primary:} 35Q30, 35B06.\\
\textbf{MSC secondary:}  76D05, 76F02.
\end{abstract}

\bigskip

The initial value problem for the Navier-Stokes system for incompressible fluids is usually written as
\begin{equation}\label{nsieq-}
\begin{cases}
 \dis\frac{\p u}{\p t}+(u\cdot \nabla) u- \nu\Delta   u=-\nabla p,\quad \dive u=0, 
 \\
 u_{\vert t=0}=u_{0}.
\end{cases}
\end{equation}
Here $u=u(t,x)$ is a time-dependent vector field on $\R^{3}$, the viscosity $\nu$ is a positive parameter
(expressed in Stokes, \ie ${\tt L}^{2}\,{\tt T}^{-1}$) and $u_{0}$ is a given divergence-free vector field. 

\medskip
In 1934, Leray \cite{MR1555394} proved the existence of global weak solutions in $L^\infty_t L^2_x \cap L^2_t \dot{H}^1_x$.
In 3D, the question of their uniqueness remains elusive and is intimately connected to deciding whether the weak
solutions enjoy a higher regularity.
Well-posedness in various function spaces has been studied thoroughly and culminates in Koch and Tataru's result
\cite{MR1808843} if the data~$u_0$~is small in the largest (\ie less constraining) 
function space (called $\operatorname{BMO}^{-1}$) that is scale and translation invariant and on which
the heat flow remains locally uniformly in $L^2_{t,x}$.

The set of singular times may or not be empty, but it is a compact subset of~$\R_+$, whose Hausdorff measure of
dimension $1/2$ is zero.
The celebrated theorem of Caffarelli, Kohn and Nirenberg \cite{CKN82} ensures that singularities form a subset of space-time
whose parabolic Hausdorff measure of dimension $1$ vanishes too (see also Arnold and Craig~\cite{AC2010}).

Note that equation~\eqref{nsieq-} corresponds to an Eulerian point of view, \ie it describes the movement of the fluid in a fixed reference frame. The natural question of tracking individual fluid particles, \ie the Lagrangian point of view, is equivalent to the existence
of a flow $\xi:\R_+\times \R^d\to\R^d$
\begin{equation}
\frac{\partial \xi}{\partial t} = u(t,\xi(t,x)), \qquad \xi(0,x)=x.
\end{equation}
The volume preserving map $\xi(t,\cdot)$ tracks the deformations of the fluid (see \eg\cite{FGT1985} and \cite{MR1354312}).

\medskip
For a comprehensive covering of most of the classical theory of Navier-Stokes, we refer the reader to, \eg[,]
Lemarié's book \cite{LR2016} and the references therein.
Berselli's recent book~\cite{Berselli2021} offers an interesting complement that blends theoretical results on the energy fluxes with
numerical methods and turbulence theory. Davidson's book~\cite{Davidson2015} provides valuable physical insight on the latter subject.

\section{Introduction}

In the next few lines, we will present a small subset of these classical results,
not necessarily in chronological order, to provide some background on the arduous question of the regularity of the solutions.
Then we will expose our own contribution, which is a new geometric approach based on the diagonalization of the $\curl$ operator.

\subsection{Classical regularity theory near a singular event}

The behavior of smooth solutions of the Navier-Stokes equation as they approach a (still conjectural) finite
blow-up time has been studied very carefully. 

\smallskip
For the $\dot{H}^1$ semi-norm, a precise rate has been known since Leray~\cite{MR1555394}:
if the first time of singularity~$T^\ast$ of a smooth solution is finite, then
\begin{equation}\label{lerayBlowupBootstrap}
\norme[L^2]{\nabla u(t)} \geq \frac{C}{(T^\ast-t)^{1/4}}\cdotp
\end{equation}
This inequality is the immediate consequence of a bootsrap of the local well-posedness result for data in~$H^1$,
when one takes into account that if $u_0$ blows up at time $T^\ast$, then $u(t)$ will blow up at time $T^\ast-t$.
Similarly, for any $0<\gamma<1/2$ and $p=3/(1-2\gamma)$:
\begin{equation}\label{sobolevBlowupBootstrap}
\norme[\dot{H}^{\frac{1}{2}+2\gamma}]{u(t)} \gtrsim \norme[L^p]{u} \geq \frac{C_\gamma}{(T^\ast-t)^{\gamma}}\cdotp
\end{equation}
The endpoint $L^\infty$ is admissible with a rate $\gamma=1/2$.

\smallskip
Thanks to the energy inequality and the Sobolev embedding, any Leray solution enjoys
a uniform control in $L^\infty_t L^2_x \cap L^2_t L^6_x$, so in particular in $L^4_t L^3_x$ and $L^{2+2/3}_t L^4_x$. 
Various authors including Foias, Guillop\'e \& Temam \cite{FGT1981}, Chemin \cite{C2004}, Cordoba, de la Llave \& Fefferman~\cite{CFL2004}
observed independently that the amplitude of Leray solutions is controlled in $L^1_{\text{loc}}(\R_+ ; L^\infty(\R^3))$ \ie
\begin{equation}\label{noSquirt}
\forall T>0, \qquad \int_0^T \norme[L^\infty]{u(t)} dt <\infty.
\end{equation}
This result is now known as the absence of squirt singularities (see \eg\cite[\S 11.6]{LR2016}).

\pagebreak
In \cite{V2010}, Vasseur proved a family of estimates in various function spaces as long as the solution remains smooth,
one of which reads
\begin{equation}\label{vasseurUnif}
\int_0^{T^\ast} \sum_{|k|\leq 2} \norme[L^1]{\nabla^k u(t)} dt
 \leq C(1+\norme[L^2]{u_0}^{4})
\end{equation}
with a constant $C$ that does not depend on the solution $u$, nor on the blow-up time $T^\ast$. 
Interpolation between~\eqref{noSquirt} and~\eqref{vasseurUnif} ensures that
\begin{equation}\label{roughBound}
\forall p\in[1,\infty], \qquad \int_0^{T^\ast} \norme[L^p]{u(t)} dt <\infty.
\end{equation}
Using the energy inequality~\eqref{roughBound} can obviously be improved to $L^q([0,T^\ast);L^p)$ with either $q=1-\frac{3}{p}$ if $p\geq 6$, or $\frac{2}{q}+\frac{3}{p}=\frac{3}{2}$ if $2\leq p\leq 6$, or $q=\frac{p}{2-p}$ if $1\leq p\leq 2$.

\medskip
These universal qualitative upper bounds are in sharp contrast with the lower bounds, which generalize~\eqref{lerayBlowupBootstrap}-\eqref{sobolevBlowupBootstrap} in the case of a finite time blow-up.
Regarding the supremum norm, \eqref{sobolevBlowupBootstrap} implies
\begin{equation}\label{LinfGrowth}
\int_0^{T^\ast} \norme[L^\infty]{u(t)}^2 dt =+\infty.
\end{equation}
In very rough terms and unless the amplitude oscillates wildly near $T^\ast$, the behavior depicted by~\eqref{noSquirt}
and~\eqref{LinfGrowth} suggests that
\[
\frac{C_\infty}{(T^\ast-t)^{1/2}} \leq \norme[L^\infty]{u(t)} \leq \frac{C(u)}{T^\ast-t}\cdotp
\]
To ``thicken'' the peaks of amplitude, one may look at uniform bounds for the heat flow, \ie Besov norms of negative regularity index.
The quantitative lower bound of Chemin \& Gallagher \cite{CG2019}
\begin{equation}
\norme[\dot{B}^{-1+2\gamma}_{\infty,\infty}]{u(t)} 
= \sup_{\tau>0} \tau^{\frac{1}{2}-\gamma} \norme[L^\infty]{e^{\tau\Delta}u(t)}
\geq \frac{C_\gamma}{(T^\ast- t)^\gamma}
\qquad (0<\gamma<1/2)
\end{equation}
is coherent with the previous intuition when $\gamma = 1/2$. The second endpoint ($\gamma=0$) requires
special care because it is also the end of the chain of critical scale-invariant spaces
\[
\dot{H}^{1/2} 
\subset L^3 
\subset \operatorname{BMO}^{-1} \subset \dot{B}^{-1}_{\infty,\infty}
\]
\ie Galilean invariant spaces $X\subset \mathcal{S}'(\R^3)$ whose norm satisfies $\norme[X]{\lambda u(\lambda x)} = \norme[X]{u(x)}$
(see Meyer~\cite{M1997} and also~\cite[Prop. 1.2]{CG2019}).
Kato's~\cite{Kato1965} and Escauriaza, Seregin \& Sver\'ak's \cite{ISS2003} theorems state that
\[
\underset{t\to T^\ast}{\operatorname{lim\, inf}}\, \norme[L^3]{u(t)} \geq c_0
\qquad\text{and}\qquad
\underset{t\to T^\ast}{\operatorname{lim\, sup}}\, \norme[L^3]{u(t)} = +\infty.
\]
Later, Seregin \cite{S2012} proved that there are no major fluctuations of the $L^3$ norm near the blow-up time, \ie
\begin{equation}\label{seregin}
\lim_{t\to T^\ast} \norme[L^3]{u(t)} = +\infty
\end{equation}
and a quantitative polylogarithmic rate was obtained recently by Tao~\cite{Tao2019}:
\begin{equation}\label{tao}
\underset{t\to T^\ast}{\operatorname{lim sup}}\, \frac{\norme[L^3]{u(t)}}{ \left( \log\log\log \frac{1}{T^\ast-t} \right)^c} = +\infty.
\end{equation}
However, a simple scaling argument (see~\cite[\S5.1]{BarkerPHD}) forces the inferior limit (over all solutions)
to be zero in~\eqref{tao}
and in any similar estimate with a diverging rate, \ie fluctuations
of the $L^3$ norm will sometimes be visible at this time-scale.
Soon afterwards, Barker \& Prange~\cite{BarkerPrange2020} investigated the possibility of reducing the length of the polylogarithm.

The well known Ladyzhenskaya-Prodi-Serrin condition reads:
\begin{equation}\label{LPS}
\int_0^{T^\ast} \norme[L^p]{u(t)}^q dt = +\infty \qquad\text{for}\qquad \frac{2}{q}+\frac{3}{p}=1, \quad p>3.
\end{equation}
Note that~\eqref{seregin} corresponds to the endpoint $p=3$, while~\eqref{LinfGrowth} matches $p=\infty$.
This second endpoint was investigated by Kozono \& Taniuchi \cite{KT00}, who even generalized it
to the (larger) BMO space.

The blow-up of scale-invariant Besov norms of negative regularity index was obtained by Gallagher,
Koch \& Planchon~\cite{GKP2016}. This lower bound implies that most supercritical norms (\ie Galilean invariant space-time function
spaces $Y$ such that $\norme[Y]{\lambda u(\lambda^2 t,\lambda x)} \leq C \lambda^{1-\gamma} \norme[Y]{u(t,x)}$ with $\gamma<1$)
will also blow up.
For the subtle behavior at the endpoint among critical spaces, \ie $\dot{B}^{-1}_{\infty,\infty}$, we refer to Cheskidov \& Shvydkoy~\cite{CS2010a} and Ohkitani~\cite{O2017}.

\medskip
Concerning spaces of higher regularity, we have known since Kato that
\begin{equation}\label{cheapBKM}
\int_0^{T^\ast} \norme[L^\infty]{\nabla u} = +\infty.
\end{equation}
The celebrated Beale-Kato-Majda criterion \cite{MR763762}, \cite{MR1953068} reads
\begin{equation}\label{BKMintro}
\int_0^{T^\ast} \norme[L^\infty]{\curl u} = +\infty
\end{equation}
and various generalizations in more involved function spaces are possible, \eg\cite{KT00},~\cite{MR1953068}.
We will briefly present Cheskidov \& Shvydkoy's \cite{CS2010b} variant of this criterion (see equation~\eqref{CS_BKM_variant} below).

\medskip
Alternatively, \textbf{each of the identities~\eqref{LinfGrowth}, \eqref{seregin}, \eqref{LPS}, \eqref{cheapBKM}, \eqref{BKMintro}
can also be stated as a regularity criterion}.
If the left-hand side integral is finite on some time interval $[0,T]$, then the corresponding solution remains regular up to
and including at time $T$, \ie one has $T^\ast >T$.
Numerical investigations of quasi-singularities are still underway, \eg by Sverak \& al.~\cite{JiaSverak}, \cite{GuillodSverak}, who possibly hint at the existence of actual singularities, or by Protas \& al.~\cite{AP2017}, \cite{KYP2020}, \cite{KP2021}, who seek flows
that maximize the growth of various norms.

\bigskip
Let us close this first panorama by a word of caution.
The heat equation is justly considered as the archetype of a well behaved parabolic regularizing model. Its solution
is indeed obtained by convolution with a Gaussian kernel
\begin{equation}
e^{t\Delta} u_0 = \int_{\R^3} u_0(x-\sqrt{t} y) W(y) dy \qquad\text{with}\qquad W(y) = (4\pi)^{-3/2} e^{-y^2/4}.
\end{equation}
However, as shown by Tychonov~\cite{tycho}, \cite{MR1998564}, this solution is not the only one if one fails to
restrict the growth of $u$ at infinity to, \eg[,] $O(e^{c x^2})$. For any $\alpha\in\R$,
the following function is a smooth (but not tempered) solution of the heat equation that coincides with $u_0$ at $t=0$:
\begin{equation}
u(t,x)+ \alpha \sum_{n=0}^\infty P_n\left({\textstyle \frac{1}{t}}\right)  e^{-1/t^2} H(t)\frac{x^{2n}}{(2n)!}\cdotp
\end{equation}
Here $u=e^{t\Delta} u_0$, $H(t)$ is the Heaviside function, $P_0=1$ and $P_{n+1}(z)=2z^3 P_n(z)+P_n'(z)$
are the polynomials involved in the computation of the $n^{\text{th}}$ derivative of $e^{-1/t^2}$.
While this type of instability may seem far from the physical range of validity of hydrodynamical models, it remains instructive.
See also \cite{MR2597507}.

\subsection{Geometric regularity theory near a singular event}

All the criteria that we have 
mentioned up to now are obviously isotropic and do not rely on any geometric structure of the flow.
There have been a few remarkable attempts to take into account the geometric nature of the Navier-Stokes equation
and we shall now present them briefly.

\medskip
A striking example of the importance of the geometry for hydrodynamics
is \textbf{turbulence}, where radically anisotropic structures (vortex filaments and pancakes)
play a central role~\cite{Davidson2015}. This observation suggests   that \textbf{the most fundamental and universal behavior of fluids
is a microlocal cascade}. However, it is equally important (and more feasible)
to describe the consequences of these interactions at intermediary scales, for example by estimating
the growth of norms of geometric quantities.

\bigskip
An important step in this direction was achieved by Constantin \& Fefferman~\cite{CF1993},~\cite{C1994},
who studied the direction of the vorticity $\omega=\curl u$, \ie[:]
\begin{equation}
\xi(t,x) = \frac{\omega(t,x)}{|\omega(t,x)|} \in \mathbb{S}^2.
\end{equation}
Using $\xi$ as a multiplier in the equation of vorticity (see~\eqref{eq_vorticity} below),
they established that
\begin{equation}
\bigl(\p_{t}+u\cdot\nabla-\nu \Delta\bigr)(\val \omega)+
\nu\val \omega\val{\nabla \xi}^{2}
=\poscal{(\omega\cdot \nabla) u}{\xi}_{\R^{3}}.
\end{equation}
Integrating over $[0,t]\times\Omega$ for $\Omega=\R^3$ or $\T^3$ leads (at first for smooth solutions,
see~\cite[eq. 20]{CF1993}) to the identity
\[
\int_\Omega \val{\omega(t,x)} dx
+\nu\int_{0}^{t}\int_\Omega \val{\omega(t,x)}\val{\nabla \xi(t',x)}^{2} dx dt'
=\int_\Omega \val{\omega(0,x)} dx +\int_{0}^{t}\int_\Omega \poscal{(\omega\cdot \nabla) u}{\xi}_{\R^{3}} dx dt'.
\]
The following geometric estimate follows immediately (using~\eqref{energy_balance}):
\begin{equation}
\int_{\Omega} \val{\omega(t,x)} dx+\nu\int_{0}^{t}\int_{\Omega} \val{\omega(t,x)}\val{\nabla \xi(t',x)}^{2} dx dt'
\le 
\norm{\omega(0)}_{L^{1}}+\frac{1}{2\nu}\left(\norm{u(0)}_{L^{2}}^{2} - \norm{u(t)}_{L^{2}}^{2}\right).
\end{equation}
To the best of our knowledge, this global $L^\infty_t L^1_x$ estimate is the only known a priori bound on the
vorticity that holds for any Leray solution, apart from the obvious $L^2_{t}L^2_{x}$ bound that follows from the energy inequality.
This result illustrates how \textbf{a local alignment in the direction of the vorticity can  deplete the nonlinearity}.
An interpretation that connects this estimate to turbulence is given in~\cite{C1994}.

\bigskip
Another notable geometric result is the one from Vasseur~\cite{V2009} on the direction of the velocity field.
This result is specific to the 3D case and can be stated as follows. If a solution blows-up at a finite time $T^\ast$,
then:
\begin{equation}
\int_0^{T^\ast} \left\| \div \left(\frac{u}{|u|}\right) \right\|_{L^p}^q dt = +\infty
\qquad\text{for}\qquad  \frac{2}{q}+\frac{3}{p}\leq \frac{1}{2}, \quad q\geq4, \quad p\geq 6.
\end{equation}
Conversely, a control of the norm implies the regularity of the solution.
This criterion is based on the incompressibility of the flow and the identity
\[
|u|\div\frac{u}{|u|} = - \Big(\frac{u}{|u|}\cdot \nabla\Big) |u|.
\]
It means that \textbf{the growth of $|u|$ along the
streamlines is linked to the divergence of the direction of $u$}.
In particular, the kinetic energy $|u|^2$ can only increase along the streamlines if they
are bent and produce some divergence in the direction of the velocity.

\bigskip
Among anisotropic criteria, let us also mention a recent result by Chemin, Gallagher \& Zhang~\cite{CGZ2018}
that investigates the possibility of detecting a singularity through one component only. If $u$ is a smooth solution
that presents a blow-up at a finite time $T^\ast$, then
\begin{equation}
\inf_{\sigma\in\mathbb{S}^2} \int_0^{T^\ast} \norme[{\dot{H}^{1/2+2/p}}]{u(t)\cdot\sigma}^p dt = +\infty
\qquad\text{for}\qquad p\geq 2,
\end{equation}
which means that all components will be affected.
They also show that
\begin{equation}
\inf_{\sigma\in\mathbb{S}^2} \sup_{t'>t}
\norme[\dot{H}^{1/2}]{u(t')\cdot\sigma} \geq C \log^{-1/2}\left(e+\frac{\norme[L^2]{u(t)}^4}{T^\ast -t}\right).
\end{equation}
The fact that the right-hand side vanishes is coherent with the remark that follows~\eqref{tao}.

\subsection{Structure of this article and summary of our results} 

Our article is structured as follows.

\smallskip
In section~\S\ref{par:geom_curl}, we expose some geometric properties of the $\curl$ operator.
The non-local diagonalization of the $\curl$ (see Lemma~\ref{lem13} and Remark~\ref{rem13}) establishes a geometrical link with
the pseudo-derivative, \ie$|D|=(-\Delta)^{1/2}$ : both operators are images of one another by a certain symmetry of the subset
of $L^2$ formed by the divergence-free vector fields. This property leads us to introduce (Definition~\ref{def:spin_definite}) the notion of \textit{spin-definite vector field}, \ie divergence-free fields such that
\[
\curl u = \pm |D| u.
\]
The end of section~\S\ref{par:geom_curl}
is dedicated to the study of such fields.
In layman's terms, fields with positive spin display an exclusive right-handedness motion at all scales, while fields with negative
spin are their chiral image in a mirror. Spin-definite fields are build as superpositions of planar Beltrami waves~\eqref{planarBeltrami2}
with independent directions of propagation and various frequencies.
Any divergence-free field can be decomposed in a unique way as the sum
of two fields with respectively a positive and a negative spin. A few numerical simulations illustrate the importance of this notion
for the description of vortex filaments.

\smallskip
The next key observation is that the non-linearity of Navier-Stokes has a cross-product structure and its weak (\ie distributional) form is a determinant (see equations~\eqref{NS_CurlForm} and \eqref{nsi++} below).
In section~\S\ref{par:twoIntegrals}, we use this geometric approach
to revisit the two classical conservation laws for Navier-Stokes, \ie the balance of energy and the balance of helicity. While the
former is constitutional of the definition of the Leray space~$L^\infty_t L^2_x \cap L^2_t \dot{H}^1_x$, we expose the latter (see equations~\eqref{conservation_NpNm} and~\eqref{helicity_balance}) as a conservation law in
\[
L^\infty_t \dot{H}^{1/2}_x \cap L^2_t \dot{H}^{3/2}_x
\]
for the spin-definite components of the flow. Our main Theorem,~\ref{mainThm1}, can be restated as follows.
\begin{thm}
In the case of a finite-time blow-up of Navier-Stokes, both of the spin-definite components of the flow
will explode simultaneously and with equal rates.
\end{thm} In simple terms, this means that \textit{singularities can only appear as the result of an unresolved
conflict of spin that escalated out of control}.
Proposition~\ref{balanceLeray}, Theorem~\ref{lem21} and~\ref{thmMain3} quantify how an imbalance between the two spins
actually prevents singularities.
In subsection~\S\ref{par:2D}, we explain how \textit{the poorer geometry of 2D flows} either enforces
the victory of one direction of rotation over the other or lets the viscosity dissolve the attempted conflict,
while the richer 3D geometry allows for the possibility of an escalation of conflicting spins.
In subsection~\ref{par:onsager_anew}, we also briefly discuss the recent developments regarding Onsager's conjecture for
the balance of energy and its counterpart for the balance of helicity.

\smallskip
Section~\S\ref{par:crit_det} pursues the geometric investigation of the weak form of the non-linearity, \ie critical determinants.
Applying this point of view to study the enstrophy produces a proof of the regularity of 2D flows based on the identity
\[
\int_{\R^3} \det(\curl u, u, -\Delta u)=0,
\]
which is valid if $u$ is a 2D divergence-free field embedded in 3D space; note that the integrand is not identically zero
and that the cancellation is the result of a space average.
More generally, we investigate (Propositions~\ref{detC12} and~\ref{detD2th} and the identity~\eqref{eq:rot_theta}) how the sign of
\[
\int_{0}^{t} \int_{\R^{3}}  \det(\curl u, u, |D|^{2\theta}u)dx dt'
\]
relates to the growth of the Sobolev norm $\dot{H}^{\theta}$ of the spin-definite components of the flow.
This analysis suggests that even though the definition of regularity is obviously local, its control
in the case of Navier-Stokes flows will likely involve \textit{non-local estimates}.
We also obtain a stability estimate among Leray solutions that has a geometric form:
\[
\norm{u_1(t)-u_2(t)}_{L^{2}}^{2} \leq \norm{u_1(0)-u_2(0)}_{L^{2}}^{2} \exp\left(
\int_0^T \frac{\norm{u_1\times u_2}_{L^2}^2}{\norm{u_1-u_2}_{L^2}^2} \right)
\]
and a variant of the Beale-Kato-Majda criterion.

\smallskip
For the convenience of the reader, Appendix~\ref{appendix} recalls the geometric proof of some vector calculus identities
whose direct computational proofs in coordinates would be non-trivial.

\section{Geometric properties of the $\curl$ operator}\label{par:geom_curl}

In this section, we collect some classical facts related to vorticity and we introduce notations, in particular
the signed decomposition of $\curl$ in~\S\ref{par:curl_signed} and the associated notion of \textit{spin}
of a 3D divergence-free field, that will be used throughout the article.

\medskip
We use the following definition for the Fourier transform on $\R^n$~:
\begin{equation}\label{fourie}
\hat u(\xi) = (2\pi)^{-n/2}\int_{\R^n} e^{-i x\cdot \xi} u(x) dx, \qquad
u(x) = (2\pi)^{-n/2}\int_{\R^n} e^{i x\cdot \xi} \hat u(\xi) d\xi.
\end{equation}
This definition provides a unitary transformation in $L^{2}(\R^{n})$.
The operator $D=-i\nabla$ satisfies
\[
\widehat{D^{\alpha}u}(\xi)=\xi^{\alpha}\hat u(\xi).
\]
The operator $|D|=(-\Delta)^{1/2}$ has symbol $|\xi|$.
We focus exclusively on the three dimensional case, \ie $n=3$,
except in the brief subsection~\S\ref{par:2D}.

\subsection{The curl operator}
\label{par:curl}

We use the notation $\rot=\curl$.
It is a Fourier multiplier of symbol
\begin{equation}\label{symbol_rot}
\rot(\xi) = \begin{pmatrix}
0&-i\xi_{3}&i\xi_{2}\\
i\xi_{3}&0&-i\xi_{1}\\
-i\xi_{2}&i\xi_{1}&0
\end{pmatrix},
\end{equation}
which is the matrix of $\eta\mapsto i\xi\times \eta$ seen as an endomorphism of the Hermitian space $\C^{3}$.
Obviously
\begin{equation}\label{}
\rot(\xi)^{*}=\overline{\tr \ { \rot(\xi)}}=\rot(\xi),
\end{equation}
which implies that the  $\curl$ is (formally) self-adjoint over $L^2(\R^3)$.
One has
\[
\rot(\xi)^{2}=-\!\!\begin{pmatrix}
0&-\xi_{3}&\xi_{2}\\
\xi_{3}&0&-\xi_{1}\\
-\xi_{2}&\xi_{1}&0
\end{pmatrix}\!\!\!\!
\begin{pmatrix}
0&-\xi_{3}&\xi_{2}\\
\xi_{3}&0&-\xi_{1}\\
-\xi_{2}&\xi_{1}&0
\end{pmatrix}\!\!=\!\!
\begin{pmatrix}
\val \xi^{2}-\xi_{1}^{2}&-\xi_{1}\xi_{2}&-\xi_{1}\xi_{3}\\
-\xi_{1}\xi_{2}&\val \xi^{2}-\xi_{2}^{2}&-\xi_{2}\xi_{3}\\
-\xi_{1}\xi_{3}&-\xi_{3}\xi_{2}&\val \xi^{2}-\xi_{3}^{2}\\
\end{pmatrix},
\]
so that 
\[
\rot(\xi)^{2}=\val \xi^{2}\Id-\xi\otimes \xi
\qquad\ie\qquad \rot^2 = \nabla\dive -\Delta.
\]
The columns of $\rot(\xi)$ are clearly orthogonal to $\xi$, which reflects the classical fact that $\div\circ\curl =0$.
The operator $|D|^{-1}\rot$ is obviously bounded on $L^2$ and is even of Calder\'on-Zygmund  type by Mikhlin's multiplier theorem.
The Leray projection onto divergence-free vector fields can be expressed in terms of $\rot$:
\begin{equation}\label{leray}
\mathbb P=\val{D}^{-2}\rot^{2}  =  \Id + \nabla (-\Delta)^{-1} \dive .
\end{equation}
The operator $\mathbb P$ is an orthogonal  projection since $\mathbb P=\mathbb P^{*}$ and $\mathbb P^{2}=\mathbb P$.
Similarly, $\Id-\mathbb P$ is an orthogonal projection onto gradient fields\footnote{As we plant our discussion exclusively
within the $L^2$ framework, there are no potential flows like $\nabla(x^2-y^2+x^3-3x z^2)$, which is both a gradient and a  divergence-free
field on $\R^3$. Such a field is formally in the range of $\mathbb{P}$.}, \ie the nullspace of $\rot$.
Note also that $\mathbb{P}$ and $\rot$ commute.

\subsection{The Navier-Stokes velocity equation in $\curl$ form}

Let us briefly present some alternative expressions of the Navier-Stokes equation involving the operator~$\rot$.
For now, we are not directly interested in the standard equation of vorticity  but rather in
expressing the linear and non-linear terms as curls.

\medskip
When $u$ is  a divergence-free vector field, the identity~\eqref{leray} implies that $\rot^2 u = -\Delta u$.
The Navier-Stokes equation \eqref{nsieq-} can thus also be written as 
\begin{equation}\label{nsieq}
\begin{cases}
\dis \frac{\p u}{\p t}+\mathbb P\bigl((u\cdot \nabla) u\bigr)+\nu \rot^{2} u=0, 
  \\
 u_{t=0}=u_{0}, \quad \dive u_{0}=0.
\end{cases}
\end{equation}
Applying $\mathbb P$ to the equation \eqref{nsieq-} leads directly to~\eqref{nsieq}.
Conversely,  \eqref{nsieq} implies that both $\p u/\p t$ and $u(t=0)$ are divergence-free, proving that $\dive u=0$.
Applying $\Id = \mathbb P - \nabla (-\Delta)^{-1} \dive$ to $u\cdot \nabla u$,
the pressure is immediately reconstructed with the identity $\nabla p=\nabla (-\Delta)^{-1} \div (u\cdot \nabla) u$.

\begin{rmk}
A common observation is that $(u\cdot \nabla)u=\sum_{j}u_{j}\p_{j}u=\sum_{j}\p_{j}(u_{j}u)=\dive (u\otimes u)$
because $u$ is a divergence-free vector field, which gives a meaning in the distributional sense  to the non-linear term as soon as~$u(t,\cdot)$ belongs to any space
embedded in $L^{2}_{\text{loc}} $.
\end{rmk}

\medskip
Let us now recall a well known identity of vector calculus. For any vector field $u$, one has~:
\begin{equation}\label{curlInProd}
(\rot u)\times u=(u\cdot \nabla)u-\frac12\nabla \val u^{2}.
\end{equation}
The first coordinate of $\rot u\times u$ is indeed:
\[
(\p_{3}u_{1}-\p_{1}u_{3}) u_{3}-(\p_{1}u_{2}-\p_{2}u_{1}) u_{2}=(u_{3}\p_{3}+u_{2}\p_{2})(u_{1})-\frac12\p_{1}(u_{3}^{2}+u_{2}^{2})
=(u\cdot \nabla)u_{1}-\frac12\p_{1}(\val u^{2}).
\]
The identity then follows by circular permutation among indices.
There is also a profound geometric reason for the above identity (known sometimes as the dot product rule),
as it is elemental in the definition of Riemannian connections (see appendix \ref{appendix}, equation~\eqref{appendix_curlucrossv}).

A direct consequence of~\eqref{curlInProd} for the non-linear term of Navier-Stokes is that
\begin{equation}\label{NS_CurlForm}
\mathbb P\bigl((u\cdot \nabla) u\bigr)= \mathbb P((\rot u)\times u).
\end{equation}
We may therefore rewrite the Navier-Stokes equation in \eqref{nsieq} as follows~:
\begin{equation}\label{nsi++}
\frac{\p u}{\p t}+\mathbb P\bigl( (\rot u)\times u\bigr)+\nu \rot^{2} u=0. 
\end{equation}
This particular form of the equation will be of central importance in what follows.
It suggests a new form of cancellations based on the following identity
\begin{equation}\label{NL_weak_det}
\poscal{(u\cdot \nabla) u}{w}_{L^{2}(\R^{3})}=\poscal{\rot u\times u}{w}_{L^{2}(\R^{3})}=\int_{\R^{3}}\det(\rot u,u,w) dx,
\end{equation}
which holds for any pair of divergence-free vector fields $u,w$.
The identity~\eqref{NS_CurlForm} thus underlines that \textbf{the non-linearity of Navier-Stokes has the structure of a cross-product},
and that its weak (distributional) form~\eqref{NL_weak_det} is a determinant that involves the vorticity, the velocity and a test function.

\medskip
Of course, this formulation is  related to the vorticity equation. One has:
\[
\rot \mathbb{P}((\rot u)\times u) = \mathbb{P} \rot((\rot u)\times u)
=\rot((\rot u)\times u) =  (u\cdot \nabla) \rot u - ((\rot u)\cdot\nabla) u.
\]
Therefore, applying $\rot$ to \eqref{nsi++} directly implies the vorticity equation
\begin{equation}\label{eq_vorticity}
\frac{\partial \omega}{\partial t} + (u\cdot \nabla) \omega  +\nu \rot^2 \omega = (\omega\cdot\nabla) u
\end{equation}
with $\omega = \rot u$. 
In the line of~\eqref{nsi++}, note that the nonlinear term of the vorticity equation inherits the structure of a cross-product:
$(u\cdot \nabla) \omega - (\omega\cdot\nabla) u = \rot(\omega \times u)$.
The vortex-stretching term
$(\omega\cdot\nabla) u = (\omega\cdot\nabla) |D|^{-2}\rot \omega$
is of order zero but highly non-local in $\omega$.
On average, it is orthogonal to $u$ (see \eqref{vortranport} below). 
The vortex-stretching term plays a central role in the cascade of energy
towards smaller scales in 3D~turbulent flows by thinning the girth of vortex tubes.

\begin{rmk}   
The Navier-Stokes equation can also be rewritten as~:
 \begin{equation}\label{}
\dis \frac{\p u}{\p t}+\mathbb P(u\cdot \nabla)\mathbb P u
+\nu \curl^{2} u=0, 
\end{equation}
to put the emphasis on the transport-diffusion aspect of the Navier-Stokes system.
However, due to the embedded pressure, the transport part is \textit{not} the divergence-free vector field $u\cdot \nabla$, but  the non-local skew-adjoint
operator
\[
\mathbb P(u\cdot \nabla)\mathbb P.
\]

For a time-independent and divergence-free vector field $U$, the flow of that operator,
\ie the solution of $\p_{t}\phi=\mathbb P(U\cdot \nabla)\mathbb P\phi$,
 is given by the Fourier Integral Operator
$$
\phi(t) = \exp^{it \, \mathbb P(U\cdot D)\mathbb P} \phi_0
$$
where $U\cdot D=-i U\cdot \nabla$; under rather mild assumptions of regularity, this operator is
self-adjoint (unbounded) on $L^{2}(\R^{3};\R^{3})$.

This flow is strikingly different from that of the vector field~$V$, \ie~$\psi(t) = \exp^{it (U\cdot D)}\psi_0$.
The difference induced by a projector ``sandwich'' is already striking among matrices.
For example, in $\R^2$, let us consider a self-adjoint matrix $A$ and a self-adjoint projection $P$ onto a
non-eigenvector of $A$:
\[
A=\begin{pmatrix}
\lambda & 0\\
0 & \mu 
\end{pmatrix}
\qquad
P=\frac{1}{2}\begin{pmatrix}
1 & 1\\
1 & 1
\end{pmatrix}= P^T
\]
then
\[
e^{itA}=\begin{pmatrix}
e^{it\lambda} & 0\\
0 & e^{it\mu} 
\end{pmatrix}
\qquad\text{while}\qquad
e^{itPAP}=\operatorname{Id}+ (e^{it\frac{\lambda+\mu}{2}}-1) P.
\]
The presence of the projector $P$ changes the evolution radically: 
the linear parts differ as $t\to0$ (the later being the $P$-projection of the former)
and the long-term behaviors are obviously completely different. 
\end{rmk}

\subsection{Decomposition of $\curl$ as a superposition of signed operators}
\label{par:curl_signed}

Let us denote by $\mathbb{P}L^2$ the subspace of $L^2(\R^3)$ composed of vector fields that are divergence-free.
As recalled in \S\ref{par:curl}, the curl operator is self-adjoint and elliptic on $\mathbb{P}L^2$.
We now want  to decompose~$\mathbb{P}L^2$ into an orthogonal direct sum of subspaces on which $\rot=\curl$ is signed.
The definition of these subspaces involves the following non-local operators associated with the ``square root'' of $\mathbb{P}$.

\begin{lemma}\label{defQ}
 One can decompose $\mathbb P=\mathbb Q_{+}+\mathbb Q_{-}$ where
 \begin{align}
\mathbb Q_{\pm}=\frac12\bigl(\mathbb P\pm \rot\val D^{-1}\bigr).
\end{align}
The operators $\mathbb Q_{\pm}$ satisfy  $\mathbb Q_{\pm}^{*}=\mathbb Q_{\pm}=\mathbb Q_{\pm}^{2}$
 and $\mathbb Q_{+}\mathbb Q_{-}=\mathbb Q_{-}\mathbb Q_{+}=0$.
\end{lemma}
\begin{proof}
The main computation is
$\mathbb Q_{\pm}^{2}=\frac14\bigl(\mathbb P^{2}+\rot^{2}\val{D}^{-2}\pm( \mathbb P \rot\val{D}^{-1}+\rot\val{D}^{-1}\mathbb P)\bigr)$.
Applying~\eqref{leray} ensures the simplifications $\mathbb P^{2}=\mathbb P$ and $[\mathbb P, \rot \val{D}^{-1}]=0$.
As $\mathbb P \rot=\rot$, we obtain $\mathbb Q_{\pm}^2=\mathbb Q_{\pm}$.
The other properties follow immediately.
\cqfd\end{proof}

\medskip
Let us define the following \textsl{signed curl} operators:
 \begin{align}
\rot_{+}=\rot\mathbb Q_{+} \quad \text{and}\quad
\rot_{-}=-\rot\mathbb Q_{-}.
\end{align}
These operators play a central role in this article.
\begin{lemma}\label{lem13}
One can decompose $\rot=\rot_{+}-\rot_{-}$. The operators $\rot_\pm$ satisfy
\begin{align}
& \rot_{\pm}^{*}= \rot_{\pm}\ge 0, \\
&\rot_{+}\rot_{-}=\rot_{-}\rot_{+}=0, \\
&\rot_{+}=\val D \mathbb Q_{+}=\mathbb Q_{+}\val D \mathbb Q_{+} = \mathbb Q_{+}\rot\mathbb Q_{+}, \\
&\rot_{-}=\val D \mathbb Q_{-}=\mathbb Q_{-}\val D \mathbb Q_{-}=-\mathbb Q_{-}\rot\mathbb Q_{-}.
\end{align}
\end{lemma}
\begin{proof}
Since $[\mathbb P, \rot]=0$ we have 
 $[\rot,\mathbb Q_{\pm}]=0$ so that 
 $
 \rot\mathbb Q_{\pm}=\rot\mathbb Q_{\pm}^{2}=\mathbb Q_{\pm}\rot\mathbb Q_{\pm}
 $.
The properties
\[
 \rot = \mathbb{P} \rot=\mathbb Q_{+} \rot+\mathbb Q_{-} \rot=\rot_{+}-\rot_{-}, \quad \rot_{+}\rot_{-}= \rot_{-}\rot_{+}=0, \quad \rot_{\pm}^{*}=\rot_{\pm}
\]
follow from the corresponding ones for $\mathbb Q_{\pm}$.
We also have 
\begin{align*}
 \rot\mathbb Q_{+}&=\frac12\bigl(  \rot\mathbb P+ \rot^{2}\val{D}^{-1}\bigr)=\frac12\bigl( \rot+\mathbb P \val{D}\bigr)=\val{D}\frac12\bigl(\mathbb P+ \rot\val{D}^{-1}\bigr)
 =\val{D}\mathbb Q_{+},
\\
  \rot\mathbb Q_{-}&=\frac12\bigl(  \rot\mathbb P- \rot^{2}\val{D}^{-1}\bigr)=\frac12\bigl( \rot-\mathbb P \val{D}\bigr)=-\val{D}\frac12\bigl(\mathbb P- \rot\val{D}^{-1}\bigr)
 =-\val{D}\mathbb Q_{-}.
 \end{align*} 
 Observing that $\val{D}\mathbb Q_{\pm}=\mathbb Q_{\pm}\val{D}\mathbb Q_{\pm}$ ensures the positivity of these
 operators.
\cqfd\end{proof}

\begin{rmk}\label{rem13}
The previous lemma ensures that the respective restrictions of $ \rot_{\pm}$
to $\mathbb Q_{\pm} L^{2}$ both coincide with~$\val D$.
The kernel of $ \rot_{\pm}$  in $\mathbb P L^{2}$ is  $\mathbb Q_{\mp} L^{2}$.
 In the orthogonal decomposition $\mathbb{P}L^2 = \mathbb Q_{+} L^{2} \oplus \mathbb Q_{-} L^{2}$,
 the matrix of the curl operator is thus
 $$
 \begin{pmatrix}
 \val D&0
 \\
 0&-\val D
  \end{pmatrix}
 $$
 \ie $\rot = |D| \circ (\mathbb{Q}_+ - \mathbb{Q}_-)$ is the \textbf{diagonalization of the curl operator}. This formula highlights
 a profound geometric connection between the curl and the pseudo-derivative $|D|$~: both operators are images of one another 
by a symmetry of $\mathbb{P}L^2$. Note also that, by functional calculus, one may define fractional operators
\begin{equation}
\rot_\pm^s = |D|^s \mathbb{Q}_\pm
\end{equation}
for any $s\in\R$; the corresponding $\rot^s = |D|^s \circ (\mathbb{Q}_+ +e^{s i\pi} \mathbb{Q}_-)$ is
however not self-adjoint if $s\in\R\backslash\Z$.
\end{rmk}

\medskip
In view of these properties, we are led to introduce the following definition.
\begin{dfn}\label{def:spin_definite}
A divergence-free vector field in $L^2(\R^3)$ is said to have \textbf{positive} (resp. negative) \textbf{spin} if it belongs
to the subspace $\mathbb{Q}_+L^2$ (resp. $\mathbb{Q}_-L^2$). We say that $u$ is \textbf{spin-definite} if
it has either positive or negative spin.
\end{dfn}
According to Remark~\ref{rem13}, a square integrable field $u$ has positive spin (up to a gradient field)
if and only if $\rot u = |D|\mathbb{P}u$ and negative spin if~$\rot u = -|D|\mathbb{P}u$.
In general, a divergence-free vector field  is not spin-definite; however, Lemma~\ref{defQ} ensures that any $\mathbb{P}L^2$ field
is always the (direct) sum of two spin-definite vector fields with opposite spins.

\begin{rmk}
The notion of spin-definite field has been known in  physics literature under the denomination
\textit{helical decomposition} and dates back to Lesieur \cite{Les1972}. It has occasionally been used in
theoretical and numerical investigations, \eg
Constantin \& Majda~\cite{CM1988}, Cambon \& Jacquin~\cite{CJ1989}, Waleffe~\cite{W1992}, 
Alexakis~\cite{A2017}. See also the discussion in~\S\ref{par:2D} below.
\end{rmk}

\begin{ex}\label{planarBeltrami}
The spin-definite fields that are spectrally supported on a sphere are examples of Beltrami flows. If $\widehat{W}(\xi)$
is a distribution supported on $\{|\xi|=\lambda\}$, then $|D|W=\lambda W$; in this case, $W$ is spin-definite if and only~if~$\rot W = \pm \lambda W$.
In the periodic setting (or if one drops the square integrability on $\R^3$), the simplest non-trivial example is of the form
\begin{equation}\label{planarBeltrami2}
W^\pm_{\lambda,\phi}(x)=\cos(\lambda x\cdot \vv{e_1}+\phi) \vv{e_2} \mp \sin(\lambda x\cdot \vv{e_1}+\phi) \vv{e_3}
\end{equation}
for some orthonormal basis $(\vv{e_1},\vv{e_2},\vv{e_3})$, a frequency $\lambda>0$
and a phase shift $\phi\in [0,2\pi)$; $W^+_{\lambda,\phi}$ is spin-positive and $W^-_{\lambda,\phi}$ is spin-negative.
The fields $e^{-\nu t \lambda^2}W^\pm_{\lambda,\phi}(x)$ are exact solutions of the Navier-Stokes equation.
They are a transient planar wave and a shear flow where the main direction of the shear
rotates (resp. right- of left-handedly) as one travels along the axis $\R\vv{e_1}$. It is the hydrodynamical equivalent of a circularly
polarized electromagnetic wave
(for further results on Beltrami flows, see~\eg\cite{MR617299} and~\cite{MR2191618}).
\end{ex}

\smallskip
Let us comment on the ``microlocal'' meaning of this definition.
It is common knowledge that all complex vector spaces (even of higher dimension) are canonically
oriented by the initial choice of one square root of $-1$ among the two choices $\pm i$. 
For $\xi\neq0$, the subspace $\xi^\perp$ of $\C^3$ is of complex dimension~2;
according to~\eqref{leray}, the matrix $|\xi|^{-1} \rot(\xi)\in\mathcal{M}_{3,3}(\C)$ defined by \eqref{symbol_rot}
is a square root of the orthogonal projector $\mathbb{P}(\xi)=I-\val{\xi}^{-2}(\xi\otimes \xi)$ of~$\C^3$ onto~$\xi^\perp$.
The pair $(\mathbb{P}(\xi),-i|\xi|^{-1}\rot (\xi))$ defines a complex structure with conjugate coordinates
$\mathbb{Q}_\pm(\xi)$.
A field has positive spin if, at each frequency $\xi\in\R^{3}\backslash\{0\}$, the complex vector $\hat{u}(\xi)$ belongs to $\range \mathbb{Q}_+(\xi)$.

\begin{lemma}
 For $\xi\in \R^{3}\backslash\{0\}$ and the matrix $\rot(\xi)\in\mathcal{M}_{3,3}(\C)$ defined by \eqref{symbol_rot}, we have 
 \begin{gather}
\ker \rot (\xi)=\C \xi,
\qquad  \range  \rot (\xi)=\xi^{\perp}=\left\{\eta \in \C^{3}\,;\, \eta\cdot \xi=0\right\}, 
\\
\operatorname{Spec} \rot (\xi)=\left\{0,\pm \val \xi\right\},
\qquad \ker(\rot (\xi)\mp \val \xi)=\range \mathbb Q_{\pm}(\xi).
\label{555}
\end{gather}
In particular, $\range \mathbb Q_{\pm}(\xi)$ is one-dimensional if $\xi\not=0$.
One has $\mathbb{Q}_-(\xi) = \mathbb{Q}_+(-\xi)=\overline{\mathbb{Q}_+(\xi)}$.
In local coordinates, the non-trivial eigenvectors are given, \eg away from the axis $\xi_2=\xi_3=0$, by:
\begin{equation}\label{eigenvector}
\delta_\pm(\xi) = \frac{1}{2|\xi|^{2}}\begin{pmatrix}
\xi_2^2 +\xi_3^2\\
-\xi_1 \xi_2 \pm i \xi_3 |\xi| \\
-\xi_1\xi_3 \mp i \xi_2|\xi|
\end{pmatrix}
\end{equation}
and one has
\begin{equation}
\ker \mathbb{Q}_\pm(\xi) = \operatorname{Span}_\C \{ \xi, \delta_{\mp}(\xi)\},
\qquad
\range \mathbb Q_{\pm}(\xi) =\C \delta_\pm(\xi).
\end{equation}
\end{lemma}
\begin{proof}
Let $\xi$ be in $\R^{3}\backslash\{0\}$.
If  for $a,b\in \R^{3}$ we have  $i\xi\times (a+i b)=0$, we obtain that $\xi\times a=\xi\times b=0$,
which is equivalent to $a\wedge \xi=b\wedge \xi=0$, i.e. $(a+ib)\in \C \xi$.
On the other hand, the two-dimensional $\xi^{\perp}$ contains the two-dimensional range of $\rot (\xi)$.
Properties \eqref{555} follow from Lemma \ref{lem13}, which implies that
\[
\rot (\xi)=\val \xi \mathbb Q_{+}(\xi)-\val \xi \mathbb Q_{-}(\xi)
\]
where $\mathbb Q_{\pm}(\xi)$ are the rank-one projections defined by
\[
\mathbb Q_{\pm}(\xi)=\frac12
\bigl(
\underbrace{I-\val{\xi}^{-2}(\xi\otimes \xi)}_{\text{real symmetric}}\pm\!\!\!
\underbrace{ \val \xi^{-1}\rot (\xi)}_{\substack{
\text{ purely imaginary}\\\text{anti-symmetric}
}
}\!\!\!\bigr).
\]
The operators $\mathbb Q_{\pm}$ are the Fourier multiplier  $\mathbb Q_{\pm}(D)$.
The formula for $\delta_\pm(\xi)$ is obtained by choosing the first column of $\mathbb Q_{\pm}(\xi)$.
\cqfd\end{proof}
\begin{rmk}
The previous choice for $\delta_\pm(\xi)$ becomes singular along the axis
$\xi_2=\xi_3=0$. To perform computations near this axis, one should instead choose
another column of $\mathbb Q_{\pm}(\xi)$ as basis vectors. 
\end{rmk}

With these local coordinates, the general expression  of the Fourier reconstruction of a divergence-free vector field is:
\begin{equation}\label{uSpin}
u(x)=\int_{\R^3} \left[ \vartheta_+(\xi) \delta_+(\xi) + \vartheta_-(\xi) \delta_-(\xi) \right] e^{ix \cdot \xi} d\xi
\end{equation}
for some spectral weights $\vartheta_\pm(\xi)\in\C$ defined almost everywhere and obtained
in a unique way by the decomposition of~$\hat{u}(\xi)$ on the
basis $(\delta_+(\xi),\delta_-(\xi))$ of $\xi^\perp$. 
As $u$ is real-valued, the weights have to satisfy
\[
\vartheta_\pm(-\xi)=\overline{\vartheta_\pm(\xi)}.
\]
One can easily compute:
\[
\rot u(x)= \int_{\R^3} |\xi|\left[ \vartheta_+(\xi) \delta_+(\xi) - \vartheta_-(\xi) \delta_-(\xi) \right] e^{ix \cdot \xi} d\xi
\]
and
\[
|D| u(x)=\int_{\R^3} |\xi|\left[ \vartheta_+(\xi) \delta_+(\xi) + \vartheta_-(\xi) \delta_-(\xi) \right] e^{ix \cdot \xi} d\xi.
\]
A field $u$ has positive (resp. negative) spin if and only if $\vartheta_-\equiv0$ (resp. $\vartheta_+\equiv0$).

\begin{cor}
The spin is a chiral notion: the mirror image of a field with positive spin by a planar symmetry of $\R^3$ is a field of negative spin.
\end{cor}
\begin{proof}
Without impeding on the generality, one may assume that  $v(x_1,x_2,x_3)=u(x_1,x_2,-x_3)$. It is then clear
from~\eqref{uSpin} and~\eqref{eigenvector} that the two fields $u$ and $v$ have opposite spins.
\cqfd\end{proof}

\begin{figure}[h!]
\begin{center}
\captionsetup{width=.9\linewidth}
\includegraphics[width=\textwidth]{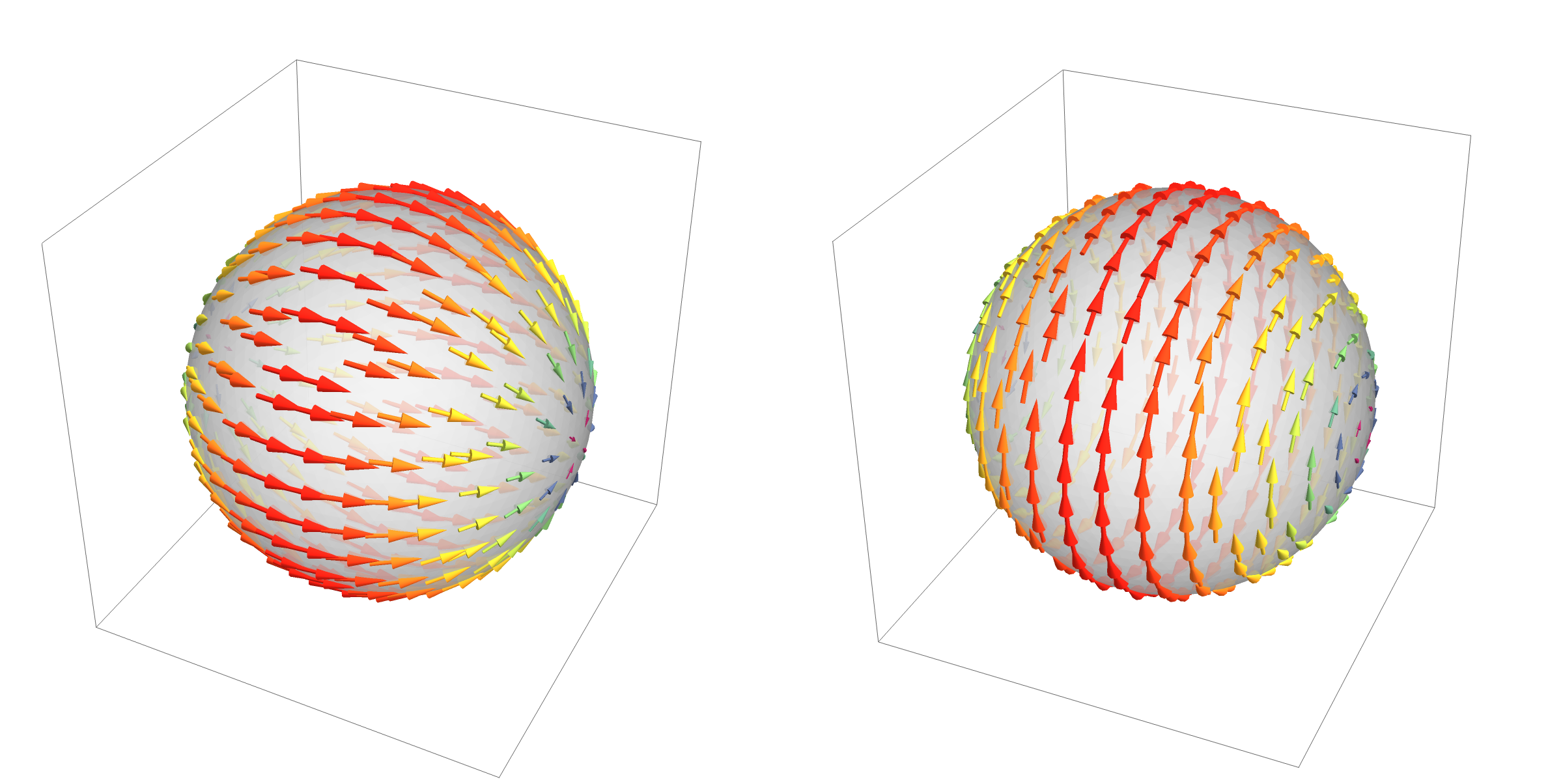}
\caption{\label{fig:delta_p}\small\it 
Real (left) and imaginary (right) parts of  $\delta_+(\xi)$ on the unit sphere $|\xi|=1$.
Multiplication by a suitable prefactor in $\C$ can rotate the axis (and the apparent singularity)
of $\delta_+(\xi)$ to any point on the sphere (the axis for the real and imaginary parts are the same).
One obtains~$\delta_-(\xi)$ by complex conjugation of~$\delta_+(\xi)$; therefore, the imaginary part of the Fourier field ``flows'' the
other way around in $\C^3$.
}
\end{center}
\end{figure}

\begin{figure}[h!]
\begin{center}
\captionsetup{width=.85\linewidth}
\includegraphics[width=\textwidth]{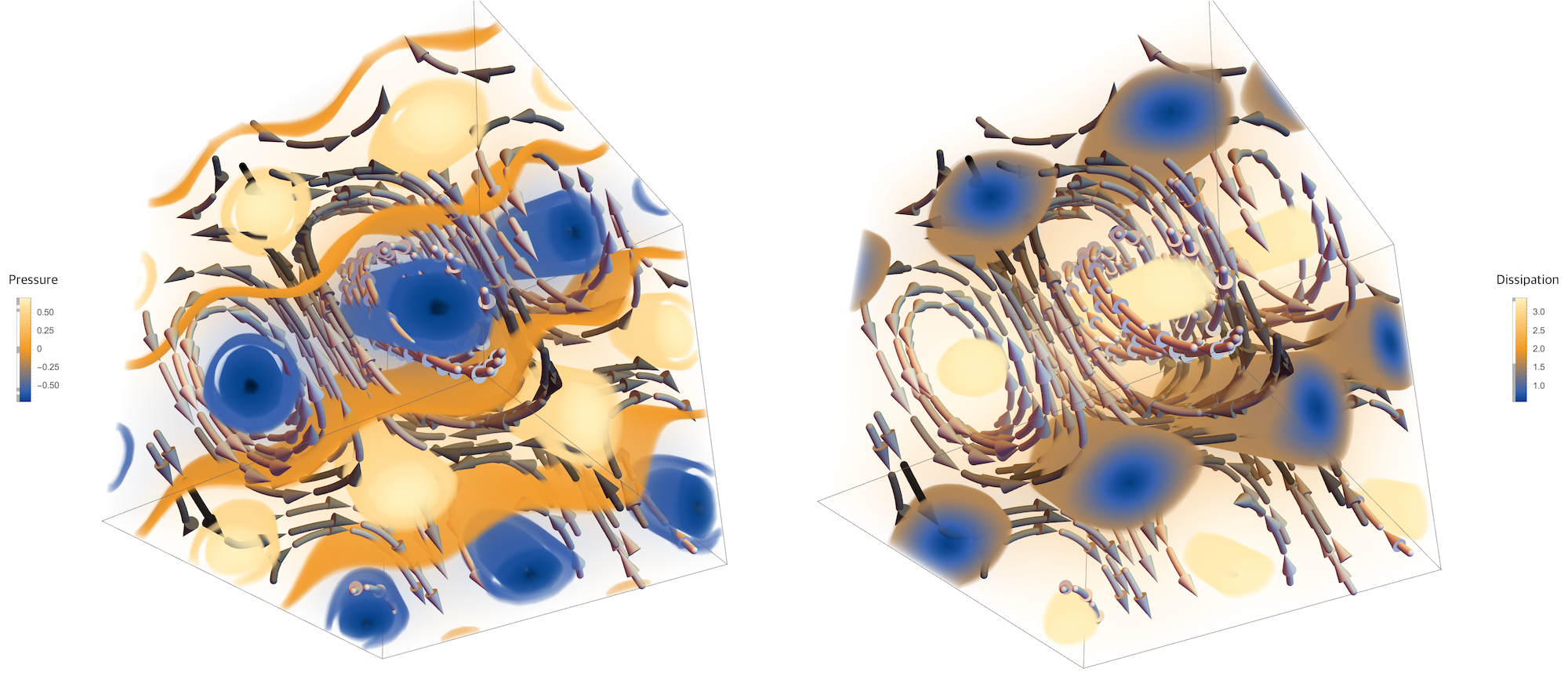}\\
\includegraphics[width=\textwidth]{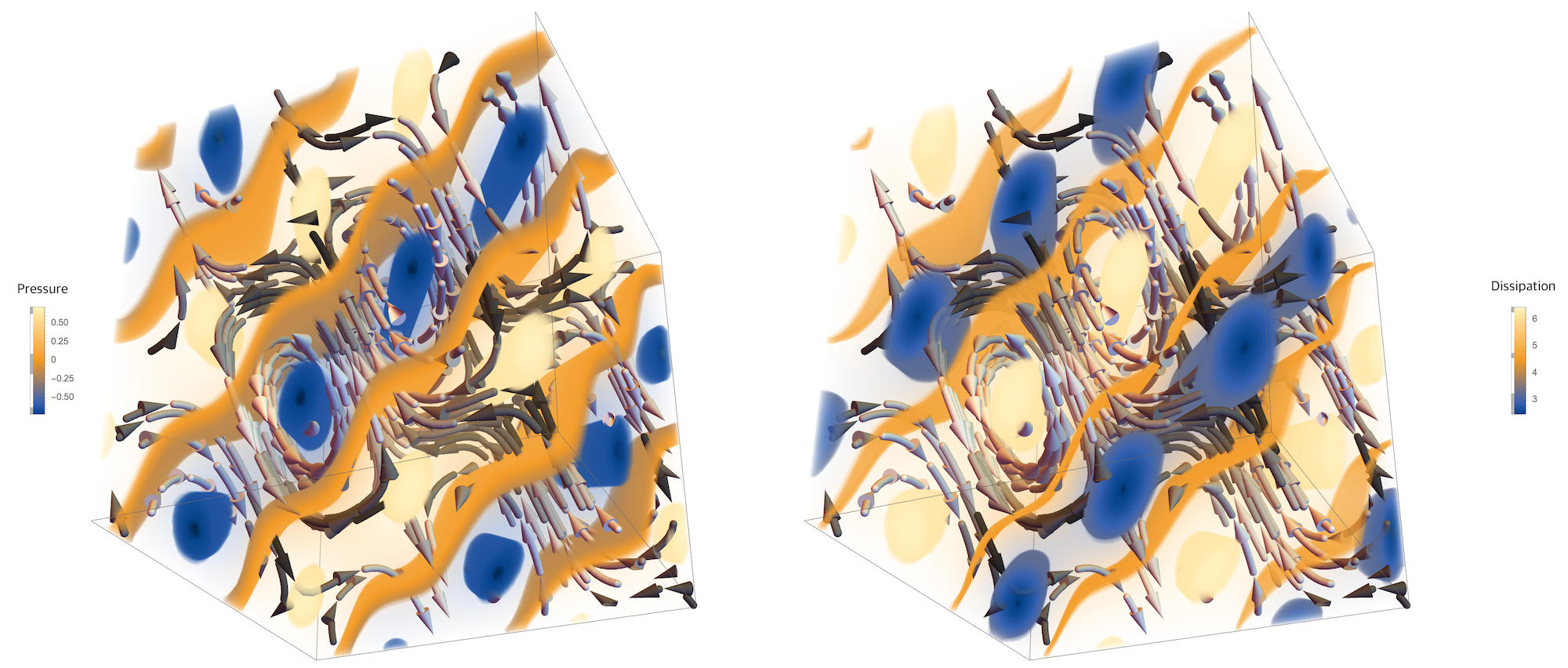}
\caption{\label{fig:pos_spin}\small\it 
Two examples of non-trivial divergence-free fields, with positive spin in the periodic setting $x\in\mathbb{T}^3$.
Above: field $u_1$ (Beltrami); below: field $u_2$ (not generalized Beltrami).
Left: streamlines of $u_j(x)$ over the pressure field. Right:  streamlines of $u_j(x)$ over the intensity of the dissipation field.
Units are arbitrary. Observe the righ-hand side motion.
}
\end{center}
\end{figure}

The family of spin-definite vector fields is quite rich and appears to have a tubular jet-structure,
where the sign of the spin reflects whether the forward motion is right- or left-handed.
For example,
\begin{gather*}
u_1(x) = -\frac{1}{2}\left( \cos(x_1-x_2) + 2\sin(x_2+x_3) \right) \vv{e_1}
- \frac{1}{2}\left( \cos(x_1-x_2) + \sqrt{2}\cos(x_2+x_3) \right) \vv{e_2}\\
+  \frac{\sqrt{2}}{2}\left( \sin(x_1-x_2) + \cos(x_2+x_3)\right)
\vv{e_3}
\end{gather*}
\begin{gather*}
u_2(x) = -\frac{1}{5}\left( 4\cos(x_1-2x_2) + 5\sin(x_2+x_3) \right) \vv{e_1}
- \frac{1}{10}\left( 4\cos(x_1-2x_2) + 5\sqrt{2}\cos(x_2+x_3) \right) \vv{e_2}\\
+  \frac{1}{10}\left( 4\sqrt{5}\sin(x_1-2x_2) + 5\sqrt{2} \cos(x_2+x_3)\right)
\vv{e_3}
\end{gather*}
are divergence-free and have positive spin \ie $\rot u_j = |D| u_j$. They are illustrated in Figure~\ref{fig:pos_spin}.

\smallskip
Note that $\rot u_1 = \sqrt{2} u_1$ so this example is a Beltrami flow; $\hat{u}_1$ is supported on the spectral sphere of radius $\sqrt{2}$.
On the contrary,~$\rot(\rot u_2 \times u_2)\neq 0$ so this second example is not even a generalized Beltrami flow; $\hat{u}_2$ involves frequencies of magnitudes $\sqrt{2}$ and $\sqrt{5}$. However, both are clearly the superposition of two planar Beltrami waves of positive spin (\ie flows from Example \ref{planarBeltrami}) that progress in different directions.
Both flows have a similar structure: they swirl in a right-hand fashion, the center of each vortex
is a zone of low pressure and high dissipation, the four hyperbolic corners of each cell (where the convection diverges) are axes of high pressure with minimal dissipation. Accounting for box-periodicity, these two examples display one single continuous vortex tube.

\medskip
If we superpose three or more planar Beltrami waves of positive spin, one can build more refined flows with positive spin
that contain an intricate network of vortex tubes.
\textbf{The positive spin imposes that the movement remains exclusively right-handed at all scales}. In the example shown in Figure~\ref{fig:pos_spin3},
four distinct regions (accounting for periodicity) of high vorticity appear to be disconnected, \ie one generates vortex tubes
of finite length.

\begin{figure}[h!]
\begin{center}
\captionsetup{width=.85\linewidth}
\includegraphics[width=\textwidth]{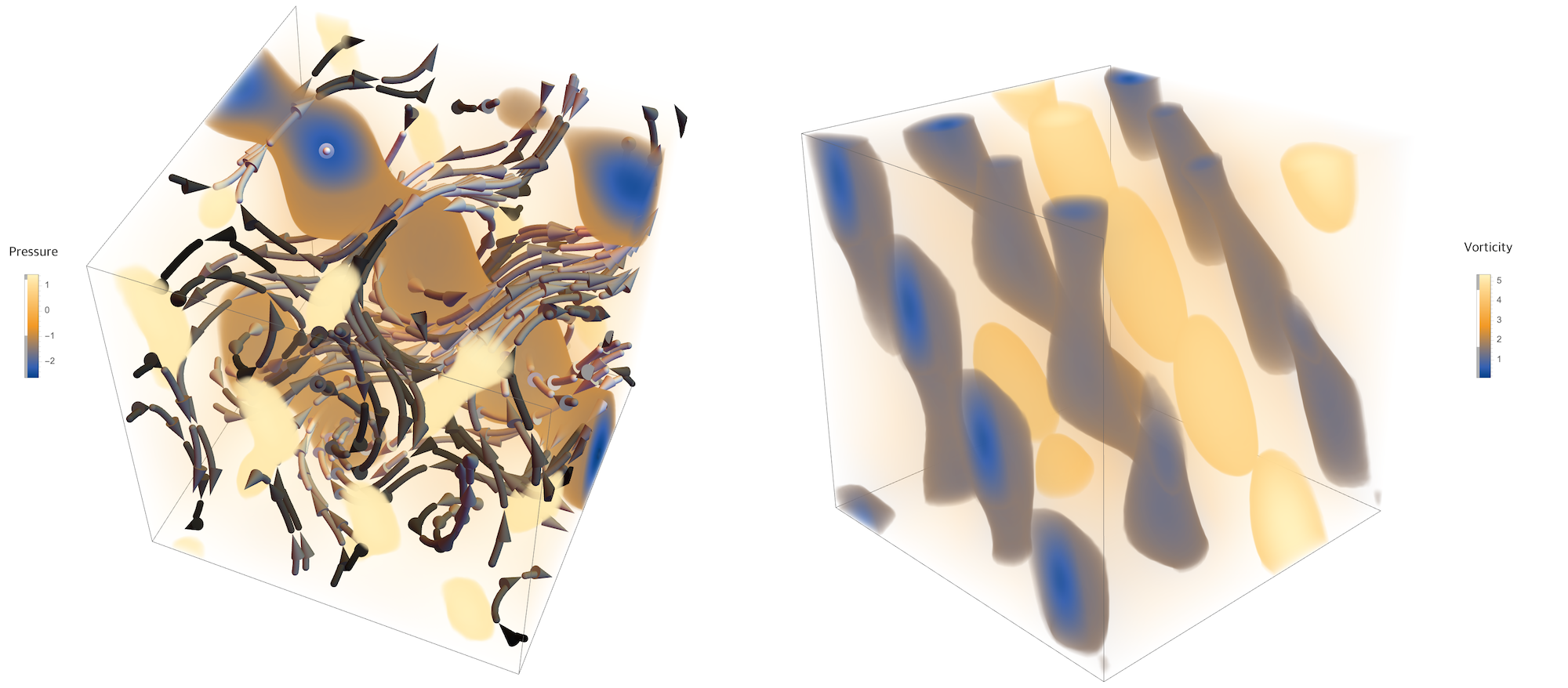}
\caption{\label{fig:pos_spin3}\small\it 
A third example of a non-trivial divergence-free field, with positive spin in the periodic setting $x\in\mathbb{T}^3$.
The field is constructed as the superposition of three planar Beltrami waves with linearly independent directions.
Left: streamlines of $u_j(x)$ over the pressure field. Right:  intensity of the vorticity field.
Units are arbitrary. The viewpoint is slightly different for better legibility. Four vortex filaments occur in the high-pressure region.
}
\end{center}
\end{figure}

These examples suggest that the family of spin-definite flows is structurally simple (superposition of planar
Beltrami waves) and yet quite rich. It is the building blocks of intricate vortex structures and deserves to be studied specifically,
as we will now do.

\begin{rmk}
The question of defining a microlocal notion of spin is legitimate\footnote{The notion of \textit{spin} introduced in this
article could then reasonably be called \textit{Fourier spin} to insist on its global nature.},
albeit non-trivial because the operators $\C\pm|D|$ are non-local.
If $u$ is a divergence-free field, there exists a stream vector (\ie vector potential) $\Psi$ such that $u=\rot \Psi$.
It is given by $\Psi=|D|^{-2}\rot u + \nabla q$ where $\nabla q$ is an arbitrary irrotational component, \eg $q=0$. If one is interested
only in the local behavior of the flow near a point $x_0\in\R^3$, one could consider a smooth cut-off
function $\chi\in \mathcal{D}(\R^3)$ supported in a ball of radius $r>0$ and such that $\chi(x)  = 1$ if $|x|\leq r/2$.
The field
\[
\tilde{u} = \rot\left( \chi(x-x_0) \Psi(x) \right) =  \chi(x-x_0) u(x) + \underbrace{\nabla  \chi(x-x_0) \times \Psi}_{\substack{\text{recirculation around}\\\text{the cutout zone}}}
\]
remains divergence-free, coincides with $u$ on the ball $B(x_0,r/2)$ and is compactly supported on $B(x_0,r)$.
The two spin-definite components of $\tilde{u}$ can be seen as a local expression of the spin of the original field $u$ near $x_0$.
However, the recirculation of $\tilde{u}$ near the edge of the cutoff zone may shadow the meaning of the spin at low frequencies,
so a secondary microlocal cutout to isolate frequencies $|\xi|\gg r^{-1}$ may be necessary.
We will not investigate this question further in this article.
\end{rmk}

\section{Two integral quantities preserved by the Navier-Stokes evolution}\label{par:twoIntegrals}

In this section, we revisit the classical energy balance for Navier-Stokes in the light of the
aforementioned properties of the $\curl$ operator over $\mathbb{P}L^2 = \mathbb Q_{+} L^{2} \oplus \mathbb Q_{-} L^{2}$.

\subsection{Classical energy method}
Leray's method was introduced in 1934 in the seminal article  \cite{MR1555394}.
It consists in  multiplying  \eqref{nsi++}  by $u$ to get
$$
\frac{d}{dt}\norm{u(t)}_{L^{2}}^{2}+2\poscal{\mathbb P(\rot u\times u)}{u}_{L^{2}}+2\nu \norm{\rot u}_{L^{2}}^{2}=0.
$$
Since $\mathbb P^{*}u=\mathbb P u=u$,  the non-linear term formally cancels out
\begin{equation}\label{formal_energy}
\poscal{\mathbb P(\rot u\times u)}{u}=\poscal{\rot u\times u}{u}=\det(\rot u, u, u)=0.
\end{equation}
This leads to the classical energy balance
$$
\norm{u(t)}_{L^{2}}^{2}+2\nu\int_{0}^{t}\norm{\!\rot u}_{L^{2}}^{2} dt'=\norm{u(0)}_{L^{2}}^{2},
$$
which, truthfully, only holds for smooth solutions in the three-dimensional case. As Leray solutions
are obtained as limits of compact sequences $(u_n)_{n\in\N}$ that satisfy the energy equality but
converge to $u$ only weakly in $H^1$, Fatou's lemma implies
\begin{equation}\label{energy_balance}
\norm{u(t)}_{L^{2}}^{2}+2\nu\int_{0}^{t}\norm{\!\rot u}_{L^{2}}^{2} dt'\leq\norm{u(0)}_{L^{2}}^{2}.
\end{equation}
The possibility of anomalous dissipation, \ie a strict inequality in~\eqref{energy_balance}, was
envisioned by Onsager~\cite{O49} and formalized \eg in~\cite{DR2000}.
Onsager's conjecture on the minimal regularity assumption on $u$ that is necessary to ensure~\eqref{energy_balance} was
solved recently by the conjonction of the works of Isett~\cite{I2018} and Constantin, Weinan \& Titi~\cite{CWT1994}.
Soon afterwards, its importance was renewed by the construction of Buckmaster, Vicol~\cite{MR3898708}
of wild (\ie non-Leray) solutions of Navier-Stokes that defy any physically reasonable energy balance (their
energy profile can even be prescribed arbitrarily)
even though they belong to a reasonable function space  $C^0_t(H^\sigma_x)$  for some $\sigma>0$, typically $\sigma\simeq 2^{-18}$.
For further details, see~\S\ref{par:onsager_anew} below.

\subsection{Conservation law associated with the signed $\curl$}\label{par:conservation}

Let us define the following quantities:
\begin{equation}\label{222}
N_{\pm}(u,t)=\norm{\rot _{\pm}^{1/2} u(t)}_{L^{2}}^{2}
+2\nu\int_{0}^{t}\norm{\rot _{\pm }^{3/2} u}_{L^{2}}^{2} dt'.
\end{equation}
Thanks to the results of \S\ref{par:curl_signed}, the sum $N_{+}(u,t)+N_{-}(u,t)$ is equivalent, for divergence-free vector fields,
to the square of the norm of $u$ in~$L^\infty_t \dot H^{1/2}_x \cap L^{2}_{t}\dot H^{3/2}_{x}$.
Inspired by the negative sign of the $\curl$ on $\mathbb{Q}_-L^2$,  let us now turn our attention to  the Krein
\cite{K1965} ``norm'' $N_{+}(u,t)-N_{-}(u,t)$.

\begin{prop}
Let $u$ be a smooth solution of \eqref{nsi++}. The following conservation law then holds:
\begin{equation}\label{conservation_NpNm}
N_{+}(u,t)-N_{-}(u,t)=N_{+}(u,0)-N_{-}(u,0).
\end{equation}
\end{prop}
\begin{proof}
Thanks to the self-adjointness of the $\curl$, one has $\poscal{\rot u}{\partial_t u}_{L^2} = \poscal{\partial_t \rot u}{u}_{L^2}$ pointwise
in time.
 Let us multiply the equation~\eqref{nsi++} by $2\rot u$. We get 
 $$
 \frac{d}{dt}\poscal{\rot u(t)}{u(t)}_{L^{2}}
 +2\poscal{\mathbb P(\rot u\times u)}{\rot u}_{L^{2}}
 +2\nu \poscal{\rot ^{3}u}{u}_{L^{2}}=0.
 $$
For smooth vector fields, the cubic term vanishes since 
 \begin{equation}\label{formal_curlergy}
 \poscal{\mathbb P(\rot u\times u)}{\rot u}=\poscal{\rot u\times u}{\rot u}=\det(\rot u, u, \rot u)=0.
 \end{equation}
The lemma then follows, with $k=1$ or $3$, from the identities $\rot^k = (\rot_+-\rot_-)^k = \rot_+^k + (-1)^k \rot_-^k$ and
\[
\poscal{\rot^k u}{u}_{L^2} = \norme[L^2]{\rot_+^{k/2} u}^2+ (-1)^k \norme[L^2]{\rot_-^{k/2} u}^2,
\]
which are a consequence of the diagonalization of the $\curl$ obtained in~\S\ref{par:curl_signed}.
\cqfd\end{proof}

\bigskip
For Leray solutions, the pendant of the conservation law~\eqref{conservation_NpNm} is not obvious.
For example, it is not clear how
$N_{+}(u,t)-N_{-}(u,t)$ compares to $N_{+}(u,0)-N_{-}(u,0)$
for all Leray solutions (see \S\ref{par:non_explosion} below).
However, if one considers the first singularity event,
the following result expresses that singularities for the 3D Navier-Stokes equation \textbf{can only occur as the
result of a direct conflict of spin}.

\begin{thm}\label{mainThm1}
If $u$ is a smooth solution of Navier-Stokes on $[0,T^\ast)$ with a maximal life-time $T^\ast <\infty$, then
\begin{equation}
  \limsup_{t\to T^\ast}{N_{\pm}(u,t)}=+\io
  \qquad\text{and}\qquad
  \lim_{k}\frac{N_{+}(u,t_{k})}{N_{-}(u,t_{k})}=1
\end{equation}
for some increasing sequence of times $t_k\to T^\ast$.
\end{thm}
An attempt at a physical interpretation of this result is proposed in~\S\ref{par:2D} below.
\smallskip
\begin{proof}
As $u$ is smooth on $[0,T^\ast)$, the conservation law~\eqref{conservation_NpNm} holds for any $t<T^\ast$ and
\[
|N_+(u,t)-N_-(u,t)|\leq C_0
\]
with \eg $C_0=\left|\int_{\R^3} \omega_0 \cdot u_0\right|$ according to~\eqref{helicity} below.
Thanks to~\cite{GKP2016}, the sum $N_+(u,t)+N_-(u,t)$ and therefore at least one of the two norms $N_\pm(u,t)$ must diverge in lim-sup as $t\to T^\ast$.
As the difference remains bounded, both norms $N_\pm(u,t)$ must diverge simultaneously.
One obtains an increasing sequence $t_k\to T^\ast$ such that
\[
N_{+}(u,t_{k})\geq C_0+k \quad\text{and therefore}\quad N_{-}(u,t_{k})\geq k.
\]
Then $|N_+(u,t_k)/N_-(u,t_k) -1|\leq C_0/k\to 0$.
\cqfd\end{proof}

\subsection{Helicity}
\label{par:helicity}

Using the properties of $\rot_\pm$ exposed in~\S\ref{par:curl_signed}, one recovers the helicity:
\begin{equation}\label{helicity}
\mathcal{H}(t) = \int_{\R^3} \omega \cdot u = \poscal{(\rot_+-\rot_-) u}{u}_{L^2} =
\norm{\rot _{+}^{1/2} u(t)}_{L^{2}}^{2} - \norm{\rot _{-}^{1/2} u(t)}_{L^{2}}^{2}.
\end{equation}
More generally, the quantity $N_+-N_-$ can be written as a conservation law for helicity:
\begin{equation}\label{helicity_balance}
N_+(u,t)-N_-(u,t)
=\int_{\R^3} \omega \cdot u - 2\nu \int_0^t  \int_{\R^3} \omega\cdot \Delta u
=\int_{\R^3} \omega \cdot u + 2\nu \int_0^t  \int_{\R^3} \nabla \omega\cdot \nabla u.
\end{equation}
The previous results imply that,
for smooth solutions of the Euler equation, $\mathcal{H}(t)$ is conserved and that for smooth solutions of Navier-Stokes,
the quantity~\eqref{helicity_balance} is invariant.
The benefit of using the non-local diagonalization of the $\curl$ operator (\ie the $\rot_\pm$ operators) is that
this new point of view isolates two distinct signed quantities $N_\pm$ in the  balance of helicity,
which is really not obvious in the right-hand side of~\eqref{helicity_balance}.
Helicity thus appears as \textbf{a measure of the balance between the spin-definite} components of $u$.

\medskip
Let us also point out that~\eqref{helicity} and Lemma~\ref{lem13} imply immediately
\begin{equation}
|\mathcal{H}(t)| \leq \norm{\rot _{+}^{1/2} u(t)}_{L^{2}}^{2} + \norm{\rot _{-}^{1/2} u(t)}_{L^{2}}^{2} = \norme[\dot{H}^{1/2}]{u}^2.
\end{equation}
One recovers the classical estimate $|\mathcal{H}(t)| = \left|\poscal{|D|^{-1/2}\rot u}{|D|^{1/2}u}\right| \leq C \norme[H^{1/2}]{u}^2$,
which relies on the fact that the operator $|D|^{-1}\rot$ is obviously bounded on $L^2$.

\begin{rmk}\label{extended_orthogonality}
Note that, contrary to the phrasing of most proofs,
the conservation of helicity does \textit{not} result from a global cancellation of terms;
instead, each term (and sub-term) given by the respective evolution equations for $u$ and $\omega$ vanishes on its own:
\[
\int_{\R^3} (\partial_t \omega) \cdot u + \nu\poscal{\nabla u}{\nabla\omega}_{L^2}
 = \poscal{(u\cdot\nabla) \omega}{u}_{L^2} + \poscal{(\omega\cdot\nabla) u}{u}_{L^2}= 0+0=0
\]
and
\[
\int_{\R^3} (\partial_t u) \cdot \omega + \nu\poscal{\nabla u}{\nabla\omega}_{L^2}
 = \poscal{(u\cdot\nabla)u}{\omega} + \poscal{\nabla p}{\omega} = 0+0=0.
\]
Indeed, using the self-adjointness of the $\curl$ twice,
the identity~\eqref{curlInProd} implies (either formally or for smooth~$u$) that,
on average, the convection term $(u\cdot\nabla)u$ is orthogonal to the vorticity:
\begin{equation}\label{convection_orth_vorticity}
\poscal{w}{(u\cdot\nabla) u}_{L^2} = \poscal{\curl u}{(u\cdot\nabla) u}_{L^2} 
=\poscal{ u}{\curl[(u\cdot\nabla) u]}_{L^2} = \poscal{ u}{\curl[\omega \times u]}_{L^2}  = \poscal{ \omega}{\omega \times u}_{L^2}=0.
\end{equation}
If $u$ is divergence-free, one has the well known identity:
\[
\poscal{\omega}{(u\cdot\nabla)u}_{L^2}+\poscal{u}{(u\cdot\nabla)\omega}_{L^2} =
-\poscal{\div u}{u\cdot \omega} = 0.
\]
Combining this last identity with~\eqref{convection_orth_vorticity}, one gets that
the transport term $(u\cdot\nabla)\omega$ is, on average, orthogonal to the velocity field:
\begin{equation}\label{transp_orth_velocity}
\poscal{u}{(u\cdot\nabla)\omega}_{L^2} = 0.
\end{equation}
Finally, as $\dive\omega=0$, and assuming enough decay at infinity:
\begin{equation}\label{vortranport}
\poscal{(\omega\cdot\nabla)u}{u}_{L^2} = \frac{1}{2} \int_{\R^3} (\omega\cdot\nabla)|u|^2
= \frac{1}{2} \int_{\R^3} \dive(|u|^2 \omega) = 0.
\end{equation}
The identities~\eqref{convection_orth_vorticity}, \eqref{transp_orth_velocity}, \eqref{vortranport} provide a simple derivation of the conservation of helicity for the Euler equation, which holds as long as it is legitimate to test the
equation for vorticity~\eqref{eq_vorticity} against $u$ itself.
\end{rmk}

\medskip
The connection with helicity provides the following uniform integral bounds
that imply that the two spin-definite components of $u$ must have, on average, a comparable size in $\dot{H}^{1/2}$.
Note however that the last result of~\S\ref{par:conservation} provides a stronger insight at the time of first singularity.
\begin{prop}\label{balanceLeray}
For any Leray solution of Navier-Stokes, one has
\begin{equation}
\int_0^\infty |\mathcal{H}(t)|^2 dt 
=
\int_0^\infty \left( \norm{\rot _{+}^{1/2} u(t)}_{L^{2}}^{2} - \norm{\rot _{-}^{1/2} u(t)}_{L^{2}}^{2} \right)^2 dt
\leq  \frac{\norme[L^2]{u_0}^4}{8\nu}
\end{equation}
and
\begin{equation}
\norme[L^4_t \dot{H}^{1/2}_x]{u}^2=
\int_0^\infty \left( \norm{\rot _{+}^{1/2} u(t)}_{L^{2}}^{2} + \norm{\rot _{-}^{1/2} u(t)}_{L^{2}}^{2} \right)^2 dt
\leq  \frac{\norme[L^2]{u_0}^2}{4\sqrt{2\nu}}.
\end{equation}
\end{prop}
\begin{proof}
The helicity is globally square-integrable in time  because
\[
\int_0^\infty |\mathcal{H}(t)|^2 dt \leq \int_0^\infty  \norme[L^2]{\omega(t)}^2 \norme[L^2]{u(t)}^2 dt
\leq \norme[L^2_tL^2_x]{\omega}^2 \norme[L^\infty_t L^2_x]{u}^2 
\leq \frac{\norme[L^2]{u_0}^4}{8\nu}\cdotp
\]
For the last step, we used the energy inequality~\eqref{energy_balance} and $ab\leq \frac{c^2}{4\beta}$
if $a,b,c,\beta > 0$ with $a+ \beta b \leq c$.
The full $L^4_t \dot{H}^{1/2}_x$ norm of $u$ is controlled by interpolation between $L^\infty_t L^2_x$ and $L^2_t\dot{H}^1_x$.
\cqfd\end{proof}

\subsection{Onsager's conjecture anew}
\label{par:onsager_anew}

In this section, we investigate briefly the minimal regularity that is required to ensure respectively the conservation
of energy or the balance of helicity.

\bigskip
Onsager's famous conjecture~\cite{O49} states that unless $u\in C^{\alpha}_x$ with $\alpha>1/3$, there may be an
energy miscount at spectral infinity and that $u$ itself is not an admissible test function.
The heuristic leading to that exponent is that the minimal regularity required to make sense of~\eqref{formal_energy}
consists in spreading one derivative across the three factors, hence the $C^{1/3}_x$ critical space. 
For the Euler equation, Constantin, Weinan \& Titi~\cite{CWT1994} indeed proved the conservation
of energy for $\alpha>1/3$ while Isett~\cite{I2018}, using convex integration,
recently showed its failure for $\alpha<1/3$ and solved the problem
that had been open for 69 years.

\smallskip
For the Navier-Stokes equation, the conservation of energy for Leray solutions was proved by Serrin under an
$L^q_tL^p_x$ assumption with $\frac{2}{q}+\frac{3}{p}=1$, $p\geq3$, which also implies smoothness (see criterion~\eqref{LPS} above).
Lions~\cite{Lions1960}, Ladyzhenskaya~\cite{Lady68} and Shinbrot~\cite{Shin74} also proved the conservation of energy when
\[
\frac{2}{q}+\frac{2}{p} \leq 1, \quad p\geq4
\]
so in particular for $L^4_tL^4_x$. This intermediary scaling ($\frac{2}{4}+\frac{3}{4}=\frac{5}{4}$) is of particular interest because it
is both too low to be a guaranteed bound for all Leray solutions, but also too high to automatically imply the smoothness of
the solution. Kukavica~\cite{Kuk2006}  weakened this assumption to a local $L^2_{t,x}$ bound on the pressure (recall that
$p$ is obtained by a Calder\'on-Zygmund operator applied to $u\times u$). The last gap in scaling was closed by Cheskidov, Friedlander \& Shvydkoy~\cite{CFS2010}, who proved that any Leray solution in $L^3([0,T];H^{5/6})$ conserves energy.
See also Leslie-Shvydkoy~\cite{LS2018b}.

To understand why the space $L^3_t\dot{H}^{5/6}_x$ is exactly
consistent with Onsager's heuristic, let us point out that, even with a loose Leibniz rule, one cannot expect to make sense of
\[
\int_{0}^T \int_{\R^3} \det(\rot u, u , u) = 0
\qquad\text{unless}\qquad
\int_{0}^T \int_{K} ||D|^{1/3} u(t,x)|^3 dxdt<\infty
\]
for any compact subset $K\subset \R^3$.
The Navier-Stokes (\ie parabolic) scaling of $L^3_t \dot{W}^{1/3,3}_x$ is $\frac{2}{3}+\frac{3}{3}-\frac{1}{3}=1+\frac{1}{3}$,
which matches that of $L^3_t\dot{H}^{5/6}_x \subset L^3_t \dot{W}^{1/3,3}_x$. The local integrability at this scale is ensured in the following way.
For a triple $s_1+s_2+s_3\geq 3/2$ with $s_j\geq0$ and at least two non-zero regularity indices and $K\subset \R^3$ bounded,
H\"older law and the Sobolev embeddings imply (see Constantin-Foias~\cite{CF88}):
\[
\int_{K} \left| (u\cdot \nabla)  v \cdot w \right| \leq
c_K \norme[L^{6/(3-2s_1)_+}]{u}  \norme[L^{6/(3-2s_2)_+}]{\nabla v} \norme[L^{6/(3-2s_3)_+}]{w}
\leq C_K \norme[H^{s_1}]{u}  \norme[H^{s_2}]{\nabla v} \norme[H^{s_3}]{w}.
\]
At Onsager's scaling, the difficulty is that, when $u\in H^{5/6}$, then $\nabla u\in H^{-1/6}$ may fail to be locally integrable.
Very elegantly,
Cheskidov, Friedlander \& Shvydkoy~\cite{CFS2010} used a frequency decomposition $u=u_{l}+u_h$
with an arbitrary spectral threshold $\kappa$ and controlled the non-trivial terms
with Bernstein's inequalities to transfer the singularity across the trilinear interaction, effectively loosening Leibniz's rule:
\begin{align*}
\int_{K} \left| (u\cdot \nabla)  u_l \cdot u \right| &\leq 
\int_{K} \left| (u_h\cdot \nabla)  u_l \cdot u_h \right| + \int_{K} \left| (u_l\cdot \nabla)  u_l \cdot u_h \right| +0
\\&\leq \norme[H^{1/2}]{u_h}^2 \norme[H^{3/2}]{u_l}  
+ \norme[H^{5/6}]{u_l} \norme[H^{1}]{u_l}  \norme[H^{2/3}]{u_h} 
\\&\lesssim \left(\kappa^{-1/3} \norme[H^{5/6}]{u_h}\right)^2 \left(\kappa^{2/3} \norme[H^{5/6}]{u_l}\right)
+ \norme[H^{5/6}]{u_l} \left(\kappa^{1/6} \norme[H^{5/6}]{u_l}\right) \left(\kappa^{-1/6} \norme[H^{5/6}]{u_h} \right)
\\&\lesssim \norme[H^{5/6}]{u}^3
\end{align*}
This computation ensures that the cancellation $\displaystyle \lim_{\kappa\to\infty}\int_{\R^3} (u\cdot \nabla)u_l \cdot u = 0$ is legitimate.

\smallskip
The other side of Onsager's conjecture for Navier-Stokes is still open. A historical breakthrough\footnote{This article
is the result of three years of reflection inspired by Vlad Vicol's remarkable talk at the CIRM of Marseille, in
December 2018, which brought the two authors together. We are grateful to Prof. Vicol for his kind advice
at that time and when we met again at the IHES in Gif-sur-Yvette in early 2020 \cite{Vicol2020}. Our meditation on the
Beltrami waves that were used in the original proof~\cite{MR3898708} ultimately led us to Definition~\ref{def:spin_definite} of \textit{spin-definite fields} and convinced us of the importance of this notion for hydrodynamics.}
was achieved very recently
by Buckmaster \& Vicol~\cite{MR3898708}, \cite{BV2019} and with Colombo~\cite{BCV2020}. They showed that a small positive
regularity $C^0_tH^\sigma_x$ with $\sigma\simeq 2^{-18}$ is not enough to prevent the existence of non-conservative viscous flows.
They constructed flows in that class whose energy profile can be prescribed arbitrarily.
Such strange flows are weak solutions of  the Navier-Stokes equation but are \textit{not} Leray solutions.
While this pathology may seem to be of a purely mathematical nature, it does have a deep connection with turbulence \cite{DLS2019}, \cite{BV2021}. These flows display a persistent low-frequency shadow of a vanishing high-frequency forcing,
which was first observed for Euler \cite{MR3374958}. This reverse cascade ends up to be stronger than what
the viscosity can diffuse. In the absence of viscosity~\cite{BMNV2021},
one can even push the regularity of the pathologies to $\sigma=1/2^-$.

\bigskip
In the same spirit as Onsager's original conjecture,
one can ask \textbf{which minimal regularity will ensure the balance of the helicity},
\ie the conservation of~$N_+-N_-$ defined above.
Roughly speaking, in order to use $\rot u$ as a test function and ensure~\eqref{formal_curlergy}, one would need
to spread two derivatives across three factors, which would place the bar at $C^{2/3}_x$.
This threshold is sometimes known as \textit{Onsager's conjecture for helicity}.
In the case of Euler's equation, Onsager's conjecture for helicity was essentially resolved  by Cheskidov, Constantin,
Friedlander \& Shvydkoy \cite{CCFS2008}.
Recently, Luigi de Rosa~\cite{LdR2019} investigated the possibility of splitting the assumption
between $u\in L^{q_1}_t C^{\alpha_1}_x$ and $\curl u\in L^{q_2}_t W^{\alpha_2,1}_x$ with $\frac{2}{q_1}+\frac{1}{q_2}=1$
and $2\alpha_1+\alpha_2\geq 1$, which suggests that, for helicity, subtle plays with scaling are possible.

\smallskip
Because of the higher regularity threshold, the estimate in the case of Navier-Stokes is simpler than the one presented above.
For example, having $u\in L^3_t(\dot{H}^{7/6}_x)$ provides enough integrability
\[
\int_0^t \int_{\R^3} \left| (u\cdot \nabla)  u \cdot \rot u \right| \leq 
\norme[L^3(L^9)]{u} \norme[L^3(L^{9/4})]{\nabla u} \norme[L^3(L^{9/4})]{\rot u}
\leq \norme[L^3(\dot{H}^{7/6})]{u} \norme[L^3(\dot{H}^{1/6})]{\nabla u}^2 \leq \norme[L^3(\dot{H}^{7/6})]{u}^3
\]
and thus legitimizes~\eqref{formal_curlergy}. 
Note that the scaling of $ L^3_t(\dot{H}^{7/6}_x)$ is consistent with $1/3$ more derivative than that of~$L^3_t(\dot{H}^{5/6}_x)$,
which was critical for the conservation of energy. This scaling is thus coherent, in spirit,
with Onsager's conjecture for helicity.
The scaling of $L^3(\dot{H}^{7/6})$ 
differs from that of $L^\infty(L^2)\cap L^2(\dot{H}^1)$ by $\frac{7}{6}-\frac{2}{3}=\frac{1}{2}$
derivative; such a control is similar in scaling to $L^\infty(\dot{H}^{1/2})\cap L^2(\dot{H}^{3/2})$ and is
therefore not known (and possibly not expected) for the most general Leray solutions.

\bigskip
\begin{rmk}
Formally, there are two other known conserved integrals for Euler and Navier-Stokes: the \textit{momentum}
\begin{equation}
P(t) = \int_{\R^3 \text{ or }\T^3} u(t,x) dx,
\end{equation}
and the \textit{angular momentum}
\begin{equation}
L(t)=\int_{\R^3\text{ or }\T^3} x\times u(t,x) dx.
\end{equation}
However, on $\R^3$, the decay of the velocity field that is necessary to define the momentum is not benign; for example, $P(t)$ is identically zero
for any integrable divergence-free field. Similarly, the weighted integrability
\[
u\in L^1((1+|x|)dx)
\]
happens to be the critical one that
cannot be propagated by the flow because the generic profile of a well localized flow decays exactly as $|x|^{-d-1}$ at infinity along most directions, which is due to the non-local effect of the pressure field (see Brandolese-Vigneron \cite{BV2007}).  Therefore $P$ and $L$ are not the most useful conservation laws
for flows on the full space $\R^3$.
\end{rmk}

\subsection{Non-explosion criteria}\label{par:non_explosion}

The Navier-Stokes system can be written for the decomposition $u=u_{+}+u_{-}$ where $u_{\pm}=\mathbb Q_{\pm}u$ are the two spin-definite components of $u$ (see Definition~\ref{def:spin_definite}):
\begin{align}\label{226}
\begin{cases}
 &\dis \frac{\p u_{+}}{\p t}+\mathbb Q_{+}\bigl(\rot u\times u\bigr)+\nu \rot_{+}^{2} u_{+}=0, \qquad u_{+}(0)=\mathbb{Q}_+u_0, 
\\[1em]
&\dis\frac{\p u_{-}}{\p t}+\mathbb Q_{-}\bigl(\rot u\times u\bigr)+\nu \rot_{-}^{2} u_{-}=0, \qquad u_{-}(0)=\mathbb{Q}_-u_0. 
\end{cases}
\end{align}
However, the coupling of the two equations through $\rot u \times u$ is highly intricate. The point of this section
is to investigate how this coupling relates to issues of regularity.

\medskip
Let us briefly explain the technical difficulty that one encounters when one attempts to generalize the conservation
law~\eqref{conservation_NpNm} to the framework of Leray solutions. 
Let $u$ be a Leray solution of Navier-Stokes with $u_0\in H^{1/2}$
and $u_k$ a sequence of Galerkine approximations of $u$ that are spin-definite.
It is common knowledge (see \eg\cite{LR2016}) that the convergence of~$u_k$ to $u$ holds in the strong topology of
$L^\infty([0,T];H^{-1}) \cap  L^2([0,T],H^s)$ for any $T>0$ and any arbitrary but fixed value $s<1$.
In particular, with $s=1/2$, one gets that
\[
\lim \poscal{\rot_\pm u_k(t)}{u_k(t)}_{L^{2}}
= \poscal{\rot_\pm u(t)}{u(t)}_{L^{2}}
\]
for almost every $t\geq0$. The proof of~\eqref{conservation_NpNm}
can be reproduced for the smooth functions~$u_k$ leading to:
\begin{align*}
\poscal{\rot_+ u_k(t)}{u_k(t)}_{L^{2}} &+2\nu\int_0^t \poscal{\rot_+ ^{3}u_k}{u_k}_{L^{2}}
+ \poscal{\rot_- u_k(0)}{u_k(0)}_{L^{2}}\\
&= \poscal{\rot_+ u_k(0)}{u_k(0)}_{L^{2}}
+\poscal{\rot_- u_k(t)}{u_k(t)}_{L^{2}}+2\nu\int_0^t \poscal{\rot_- ^{3}u_k}{u_k}_{L^{2}}
\end{align*}
\ie $N_+(u_k,t)+N_-^0 =N_+^0+ N_-(u_k,t)$.
However, in general, Fatou's lemma can only guarantee that
\[
N_\pm(u,t) \leq \underset{{k\to\infty}}{\lim\inf} \, N_\pm( u_k,t),
\]
which is not in our favor if we want to pass to the limit in the previous identity. 

\bigskip
It is possible to circumvent this difficulty in the case of spin-definite solutions.

\begin{thm}\label{lem21}
If $u$ is a Leray solution of Navier-Stokes stemming from $u_0\in H^{1/2}$,
then $u$ is smooth as long as it remains spin-definite.
\end{thm}

\begin{proof}
Let $u$ be a Leray solution of Navier-Stokes with $u_0\in H^{1/2}$ and $T_1>0$ such that $u$ is
spin-definite on $[0,T_1]$. Without impeding the generality, one can apply a planar symmetry if necessary and assume positive spin.
It is common knowledge that $u$ is smooth on some non-trivial interval $[0,T_2]$. One considers
\[
T=\sup\left\{ t\in[0,T_1]\,;\, u\text{ is smooth on }[0,t]\right\}\geq T_2.
\]
Reasoning by contradiction, let us assume that $T\leq T_1$. Then~\eqref{conservation_NpNm} and the fact that $u$ has
positive spin imply $N_{+}(u,t)=N_{+}(u,0)$ for all $t\in[0,T)$ \ie 
\[
\norm{u(t)}_{\dot{H}^{1/2}}^{2}
+2\nu\int_{0}^{t}\norm{u}_{\dot{H}^{3/2}}^{2} dt'
=
\norm{\rot _{+}^{1/2} u(t)}_{L^{2}}^{2}
+2\nu\int_{0}^{t}\norm{\rot _{+}^{3/2} u}_{L^{2}}^{2} dt'
=\norm{u_0}_{\dot{H}^{1/2}}^{2}.
\]
In particular, $u\in L^\infty([0,T);\dot{H}^{1/2})$ and~\cite{GKP2016} implies that $T$ cannot be a singular time;
consequently, one has~$T> T_1$.
\cqfd\end{proof}
\begin{rmk}
According to~\eqref{226}, a solution $u$ remains spin-definite if and only if $\mathbb{P}(\rot u\times u)$ has the same spin as $u$.
In general, it is not clear that this property is propagated by the flow. At least, this is the case for generalized Beltrami flows,
\ie when $\rot(\rot u\times u)=0$ because then $\mathbb{P}(\rot u\times u)=0$.
\end{rmk}

\medskip
The assumptions of the previous statement are somewhat exhorbitant.
In the rest of this section, we investigate instead how the respective sizes of the spin-definite components $\mathbb{Q}_\pm u$
of a smooth solution $u$ are related to the emergence of singularities.
However, as we only need smoothness to ensure the conservation of~$N_+(u,t)-N_-(u,t)$, we will preserve some
generality by assuming instead that $u$ is a Leray solution such that $|N_+(u,t)-N_-(u,t)|\leq C_0$.

\smallskip
The following lemma will be useful to bound a pair of close numbers from a common lower bound.
\begin{lemma}
For $\alpha, \beta\in\R_{+}$ and positive $C_{j},\varepsilon_{j}$, we have:
\[
C_{0}\ge \val{\alpha-\beta}\ge -C_1+ \varepsilon_{1}\min(\alpha^{\varepsilon_{2}}, \beta^{\varepsilon_{2}})
\qquad\Longrightarrow\qquad
\max(\alpha, \beta)\le C_{0} + \bigl( \varepsilon_{1}^{-1}(C_{0}+C_1)\bigr)^{1/\varepsilon_{2}}
\]
and
\[
C_{0}\ge \val{\alpha-\beta}\ge -C_1+ \varepsilon_{1}\min(\log\alpha, \log\beta)
\qquad\Longrightarrow\qquad
\max(\alpha, \beta)\le C_{0} + \exp\bigl( \varepsilon_{1}^{-1}(C_{0}+C_1)\bigr).
\]
\end{lemma}
\begin{proof}
Since the assumption and the conclusion are symmetrical in $\alpha, \beta$, we may assume that $0\le \beta\le \alpha$.
We then have $C_{0}+\beta+C_1\ge \alpha+C_1\ge\beta+ \varepsilon_{1}\beta^{\varepsilon_{2}}$ 
so
\[
\varepsilon_{1}\beta^{\varepsilon_{2}}\le C_{0}+C_1
\qquad\ie\qquad
\beta\le\bigl( \varepsilon_{1}^{-1}(C_{0}+C_1)\bigr)^{1/\varepsilon_{2}}.
\]
Consequently,
$\max(\alpha, \beta)=\alpha\le \beta+C_{0}\le \bigl( \varepsilon_{1}^{-1}(C_{0}+C_1)\bigr)^{1/\varepsilon_{2}}+C_{0}$.
The second claim can be obtained in a similar way.
\cqfd\end{proof}

\medskip
At a time of first singularity, we have already mentioned (see the last result of~\S\ref{par:conservation})
that $N_+(u,t)$ and $N_-(u,t)$ will simultaneously diverge to $+\infty$ and at the same rate. 
As Leray's flow goes on, the value
of $N_+(u,t)-N_-(u,t)$ may be altered through each singular event.
If the conflicts of spins were resolved (possibly in a non-unique way) by favoring one over the other, this could lead to a
substantial drift.  The following result quantifies, in this general setting, that even a logarithmic in-balance between
the two  spins is enough to deter singularities.
\begin{thm}\label{thmMain3}
If $u$ is a Leray solution of Navier-Stokes such that
\begin{equation}\label{229}
\forall t\in [0,T),\qquad
C_0 \geq \val{N_{+}(u,t)-N_{-}(u,t)}\ge -C_{1} + \varepsilon\min\bigl(\log N_{+}(u,t), \log N_{-}(u,t)\bigr)
\end{equation}
for some constants $C_0, C_1,\varepsilon>0$. Then $u$ remains smooth on $[0,T^\ast)$ with $T^\ast > T$. 
\end{thm}
\begin{proof}
The previous lemma implies 
$N_\pm(u,t)\le  \exp\bigl[\varepsilon_{1}^{-1}\left(C_0+C_1\right)\bigr]$ on $[0,T)$
and in particular
\[
\norme[{L^\infty([0,T];\dot{H}^{1/2})}]{u}^2 \leq \sup_{t\in[0,T]} N_+(u,t)+N_-(u,t)
\leq 2\exp\bigl[\varepsilon_{1}^{-1}\left(C_0+C_1\right)\bigr]
\]
thus $u(T)$ is smooth thanks to~\cite{GKP2016} and the solution can be extended slightly beyond that point by the standard
local well-posedness argument.
\cqfd\end{proof}

\subsection{Comparison with the dimension $n=2$}\label{par:2D}
Let us conclude this section by a brief investigation of the case of dimension 2.
The general expression of a divergence-free real-valued 2D vector field is
\begin{equation}\label{uDontSpin}
\vv{u}(x)=\frac{1}{2\pi} \int_{\R^2} \vartheta(\xi) \vv{\delta(\xi)} e^{i x\cdot \xi} d\xi
\end{equation}
where $\vv{\delta(\xi)}=\xi^\perp/|\xi|\in\R^2$ and $\vartheta(\xi)\in\C$ satisfies $\vartheta(\xi)=-\overline{\vartheta(\xi)}$; note the
anti-Hermitian symmetry because of the anti-symmetric nature of $\vv{\delta(\xi)}$ in 2D. Exceptionally, we write the arrows as a
visual cue to distinguish between vector and scalar quantities.
It is obvious that
\begin{equation}
|D|\vv{u} = \frac{1}{2\pi}\int_{\R^2} |\xi|\vartheta(\xi) \vv{\delta(\xi)} e^{i x\cdot \xi} d\xi,
\qquad
\omega = \curl \vv{u} = \frac{1}{2\pi}\int_{\R^2} i|\xi|\vartheta(\xi) e^{i x\cdot \xi} d\xi.
\end{equation}
Even though $|D|\vv{u}\in\R^2$ is non-local while $\omega \in\R$ and is local,
on the spectral side, the two operators are conjugate of one another:
\begin{equation}
\omega = i\vv{\delta(D)}\cdot |D|\vv{u} \qquad\text{and}\qquad |D|\vv{u}  = -i \vv{\delta(D)}( \omega ).
\end{equation}
This property means that, in 2D, the structure of the $\curl$ is not as rich as its 3D analog (compare with Remark~\ref{rem13})
and that, consequently, the conflict of two 2D contra-rotating vortices is not as profound as a conflict of spins in~3D.

\smallskip
In 2D, the resolution of such a conflict can only lead to a plain redistribution of the amplitude  $\vartheta(\xi)$ in~\eqref{uDontSpin}
and as  the geometry of the equation does not leave any room
for microlocal compensations, the flow either ``has to'' make a choice in favor of one direction of rotation
or, in the case of perfect balance, let the viscosity eat up the singularity attempt.
As we know, the Navier-Stokes equation is well-posed in 2D and the qualitative
behavior of the vorticity \cite{GW2006} matches this heuristic.

In 3D, a redistribution among the pair of amplitudes $(\vartheta_+(\xi),\vartheta_-(\xi))$ in~\eqref{uSpin}
also means favoring one spin over the other. However, the richer  geometry provides the flow with a new way
of ``not chosing'': it can amplify both spins simultaneously instead of letting the viscosity take over, which results in an escalating  conflict
of spin.
Singularities, if they occur, are thus \textbf{the byproduct of this unresolved microlocal game of chicken}.

\smallskip
Of course, in the physical realm, the presence of sticky boundaries (\ie with Dirichlet conditions) can produce numerous
cases of spin imbalance and couplings, which gives the boundary layer its driving role in turbulence,
regardless of whether or not true singularities or only quasi-singularities occur.
One also has to wonder whether the late resolution of physically admissible extreme events of this type
(\ie conflict of spins that have escalated for a long time) favors \textbf{subsequent cancellations},
which could be the mechanism that drives intermittency.

\begin{rmk}
We encourage the reader to consider the recent numerical simulations of Alexakis~\cite{A2017}.
Our colleagues in physics
study the energy and helicity fluxes of turbulent flows, by decomposing the influence of all possible interactions
among spin-definite components. The \textbf{numerical evidence} hints at multiple non-trivial facts:
the total energy flux can be split into three spin-related fluxes that  remain independently constant in the
inertial range; one of them amounts to 10\% of the total energy flux and is a (hidden) \textit{backwards} energy cascade,
which subsists even in fully developed 3D turbulence.
The helicity flux can be decomposed in a similar fashion into two fluxes that
remain constant in the inertial range.
\end{rmk}

\section{Critical determinants and non-local aspects of the regularity theory}\label{par:crit_det}

In this section, we investigate the idea of computing energy estimates for $\rot^\theta u$ with various values of $\theta>0$.
Each computation leads to a determinant whose average sign plays a key role both in the growth of the regularity norms
in the case of a potential blow-up and in their control as long as the flow remains smooth.
It is worth insisting on the fact that \textbf{geometric and non-local estimates seem to play a central role in the question of the regularity}
of the solutions of 3D Navier-Stokes.
This study also leads to a geometric criterion for the uniqueness of Leray solutions
and a slight variant of the Beale-Kato-Majda criterion.

\subsection{Example: a geometric drive for enstrophy}

Let us first investigate the well-known case of the enstrophy
\begin{equation}
\mathcal{E}(t)=\int_{\R^3} |\nabla u|^2 = \int_{\R^3} |\omega|^2.
\end{equation}
The equivalence between the two formulations follows \eg from $\rot^2 = -\Delta$ for divergence-free fields.

\medskip
Assuming regularity, one uses $\omega$ as a test function in the vorticity equation~\eqref{eq_vorticity}
and takes advantage of the cross-product structure of the nonlinearity, \ie
$(u\cdot \nabla) \omega - (\omega\cdot\nabla) u = \rot(\omega \times u)$.
One is led to the following balance:
\begin{equation}\label{enstrophyBalance}
\norm{\omega(t)}_{L^{2}}^{2}
 +2\nu\int_0^t \norm{\rot \omega(t')}_{L^{2}}^{2} dt'
 +2\int_0^t \int_{\R^{3}}\det\bigl(u,\rot u,  \Delta u\bigr) dx dt' =
 \norm{\omega_0}_{L^{2}}^{2}.
\end{equation}
Note that 
$\poscal{\omega\times u}{\rot \omega}_{L^{2}}=\int_{\R^{3}}\det\bigl(u,\rot u,  \Delta u\bigr) dx$. This computation
is a typical example involving a \textit{critical determinant}: the average sign of the determinant is responsible for the variations
of the norms measuring the regularity of the flow, here in terms of enstrophy.
When $\nu=0$, \ie for (smooth) 3D Euler flows, the space-time average of $\det\bigl(u,\rot u,  \Delta u\bigr)$ is
the sole geometrical drive of the variations of enstrophy.

The best known a priori upper bounds for enstrophy is a Riccati-type control by Lu \& Doering~\cite{LD2008}:
\begin{equation}
\mathcal{E}'(t) \leq C \mathcal{E}^3(t)
\end{equation}
It is obtained by estimating the critical determinant mentioned above and diverges in finite time.
For advanced numerical experiments on the growth of enstrophy for 3D viscous flows,
see \eg\cite{AP2017}, \cite{KYP2020}, \cite{KP2021} and the numerous references to the numerical literature therein.

An immediate corollary of~\eqref{enstrophyBalance} is a geometric criterion for regularity:
\begin{equation}\label{regC1}
\int_0^T\int_{\R^{3}}\det\bigl(u,\rot u,  \Delta u\bigr) dx dt \geq 0 \qquad\Longrightarrow\qquad
u \in L^\infty([0,T];\dot{H}^1) \cap L^2([0,T];\dot{H}^2).
\end{equation}
For example, one recovers in this manner that all irrotational flows are smooth because the critical determinant
vanishes identically (they are indeed the gradients of solutions of the heat equation).


\medskip
As a slightly more involved application, let us investigate the case of 3D fields with 2D symmetry, \ie
\[
v=\begin{pmatrix}v_{1}(x_{1}, x_{2})\\v_{2}(x_{1},x_{2})\\0\end{pmatrix}
\qquad\text{and}\qquad
\omega=\begin{pmatrix}0\\0\\\p_{1}v_{2}-\p_{2}v_{1}\end{pmatrix}.
\]
For such a field, one has
\[
\det(v, \rot v, \Delta v)=-\left\vert\begin{matrix}0&v_{1}&\Delta v_{1}\\0&v_{2}&\Delta v_{2}\\\p_{1}v_{2}-\p_{2}v_{1}&0&0\end{matrix}
\right\vert=(\p_{2}v_{1}-\p_{1}v_{2})\bigl(v_{1}\Delta v_{2}-v_{2}\Delta v_{1}\bigr).
\]
If one introduces the stream function $\psi(x_1,x_2)$ such that $v_{1}=\p_{2}\psi$ and $v_{2}=-\p_{1}\psi$:
\[
\det(v, \rot v, \Delta v) = \Bigl(-(\p_{2}\psi)(\p_{1}\Delta \psi)
+(\p_{1}\psi)(\p_{2}\Delta \psi)\Bigr)\Delta \psi,
\]
which has no particular reason to vanish but leads to a \textbf{global cancellation} for any $t>0$:
\begin{align*}
\int_{\R^3} \det(v, \rot v, \Delta v) &= 
\frac{1}{2}\left( \poscal{\p_{2}\psi}{-\p_{1}(\Delta \psi)^{2}}_{L^2}
+\poscal{\p_{1}\psi}{\p_{2}(\Delta \psi)^{2}}_{L^2} \right)
\\
&=\frac{1}{2}\poscal{\p_{1}\p_{2}\psi - \p_{2}\p_{1}\psi}{(\Delta \psi)^{2}}_{L^2} =0.
\end{align*}
In particular,~\eqref{regC1} implies the global regularity of such solutions, which has been known since Leray~\cite{MR1555394}.
As the balance law~\eqref{enstrophyBalance} also holds for smooth solutions of the Euler equation, the previous
computation implies the conservation of enstrophy for smooth 2D Euler flows.

\begin{rmk}
For a general 3D divergence-free flow $u$, invoking the vector potential $u=\rot \Psi$ and computing the critical
determinant in~\eqref{enstrophyBalance}
brings out \textit{288 terms} involving the product of a first, second and third order derivative of the components of $\Psi$,
with no obvious compensations through space averages. This remark illustrates the huge gap in complexity
between 2D and 3D flows.
\end{rmk}

\subsection{General case}
Let us go back to the Navier-Stokes equation written in the form~\eqref{nsi++}. The weak form
of the nonlinear term is, as mentioned in~\eqref{NL_weak_det}, a determinant:
\begin{equation}
\poscal{\partial_t u}{w}_{L^2} + \nu \poscal{\rot^2 u}{w}_{L^2} + \int_{\R^{3}}\det(\rot u,u,w) dx = 0
\end{equation}
for all divergence-free test fields $w$.
Assuming $u$ is smooth, one can collect various balance laws for Navier-Stokes by choosing $w$ appropriately as a function of $u$.
The two standard choices are either $w=u$, which gives Leray's energy equality for smooth solutions, and $w=\rot u$, which
we explored in~\S\ref{par:conservation} and which relates to the balance of helicity.
Taking $w=\rot^2 u$ leads to~\eqref{enstrophyBalance} and the balance of enstrophy with, this time, a non-trivial  critical determinant.

\bigskip
Leray's energy identity can be extended given a first integral of the flow, \ie $\alpha$ such that $(u\cdot \nabla) \alpha=0$.
Then $w=\alpha u$ is a divergence-free field and we have 
\[
\poscal{\p_{t}u}{\alpha u}_{L^2}+\nu \poscal{\rot^{2}u}{\alpha u}_{L^2}=0,
\]
and thus
\begin{equation}
\poscal{\alpha u}{u}_{L^2}-\int_{0}^{t}\poscal{\dot \alpha u}{u}_{L^2} +2\nu\int_{0}^{t}
\poscal{\rot^{2}v}{\alpha u}_{L^2}=\poscal{\alpha(0) u_0}{u_0}_{L^2}.
\end{equation}
For example, with $\alpha(t)= e^{-2 \lambda t}$, we get a family of conservation laws indexed by $\lambda>0$:
\begin{equation}
e^{-2\lambda t}\norm{u(t)}_{L^2}^{2}+2\lambda \int_{0}^{t}e^{-2\lambda t'}\norm{u(t')}_{L^2}^{2}dt'+
2\nu\int_{0}^{t}e^{-2\lambda t'}\norm{\nabla u (t')}_{L^2}^{2}dt'=\norm{u_0}_{L^2}^{2},
\end{equation}
which is a weighted time-integral (gauge transform) of the classical energy balance that puts $t' \sim 1/(2\lambda)$ into focus.
Similarly, for $w=e^{-2 \lambda t} \rot u$, one gets a variant of~\eqref{conservation_NpNm}:
\begin{equation}
e^{-2 \lambda t}\poscal{u}{\rot u}_{L^2}
+2\lambda \int_0^t e^{-2 \lambda t'} \poscal{u}{\rot u}_{L^2}
+ 2\nu \int_0^t e^{-2 \lambda t'} \poscal{\rot^2 u}{\rot u}_{L^2} = \poscal{u_0}{\rot u_0}_{L^2}.
\end{equation}
Note that $\poscal{u}{\rot u}_{L^2} = \norm{\rot_+^{1/2} u}^2_{L^2}-\norm{\rot_-^{1/2} u}^2_{L^2}$
and $\poscal{\rot^2 u}{\rot u}_{L^2} = \norm{\rot_+^{3/2} u}^2_{L^2}-\norm{\rot_-^{3/2} u}^2_{L^2}$.

\bigskip
Let us now investigate the more interesting case where $w=\rot_\pm u$.
\begin{prop}\label{detC12}
If $u$ is a smooth solution of Navier-Stokes, one has the following balance laws:
\begin{equation}
 \norm{\rot_{\pm}^{1/2} u(t)}^{2}+2\nu\int_{0}^{t}\norm{\rot_{\pm}^{3/2} u}^{2}dt'
 +\int_{0}^{t} \int_{\R^{3}}  \det(\rot u, u, |D|u)dx dt'
 =\norm{\rot_{\pm}^{1/2}u_0}^{2}.
\end{equation}
Note that the critical determinant is identical in both cases, which is a new proof of~\eqref{conservation_NpNm}.
One has also:
\begin{equation}
 \norm{u(t)}_{\dot{H}^{1/2}}^{2}+2\nu\int_{0}^{t}\norm{u(t')}_{\dot{H}^{1/2}}^{2} dt'
 +\int_{0}^{t} \int_{\R^{3}}  \det(\rot u, u, |D|u)dx dt'
 =\norm{u_0}^{2}_{\dot{H}^{1/2}}.
 \end{equation}
\end{prop}
\begin{proof}
The only non-trivial point is the critical determinant. One has:
\[
\det(\rot u, u, \rot_{+}u)dx
= \det(\rot u, u, \rot u+\rot_{-}u)  =  \det(\rot u, u, \rot_{-}u)
\]
and thus
\[ 
\det(\rot u, u, |\rot| u)dx = 
\det(\rot u, u, \rot_+ u) + \det(\rot u, u, \rot_- u)
=2\det(\rot u, u, \rot_\pm u).
\]
Finally, since $\val \rot=\val D\mathbb P$, we can replace $|\rot| u$  by $|D|u$. Subtracting the two identities gives~\eqref{conservation_NpNm}, while adding them up provides the last claim.
\cqfd\end{proof}
\begin{rmk}
Thanks to Lemma~\ref{lem13}, one can rewrite this critical determinant as:
\begin{equation}
\det(\rot u, u, |D|u) = \det((\rot_+ - \rot_-) u, u, (\rot_+ + \rot_-) u) = 
-2\det(u, \rot_+ u, \rot_- u).
\end{equation}
This determinant is the geometrical drive for the growth of the $\dot{H}^{1/2}$ norm.
Among possible cancellations, it vanishes for Beltrami waves ($\rot u$ proportional to $u$), for flows spectrally supported on a sphere ($|D|u$ proportional to $u$) and, most importantly, for spin-definite flows ($\rot u$ proportional to $|D|u$).
\end{rmk}

\bigskip
To handle fractional powers, it is simplest to split the spin-definite components to avoid problems with the
lack of self-adjointness. Using~$w=\rot_\pm^{2\theta} u$ for some~$\theta>0$ and the properties established
in~\S\ref{par:curl_signed}, one gets:
\begin{equation}\label{eq:rot_theta}
\norme[L^2]{\rot_\pm^{\theta} u(t)}^2
+ 2\nu \int_0^t \norme[L^2]{\rot_\pm^{\theta+1} u(t')}^2 dt'
+ 2\int_0^t\int_{\R^{3}}\det(\rot u,u,\rot_\pm^{2\theta} u) dx dt' = 
\norme[L^2]{\rot_\pm^{\theta} u_0}^2.
\end{equation}
This time, the cancellation takes the form:
\[
\det(\rot u,u,\rot_+^{2\theta} u) + \det(\rot u,u,\rot_-^{2\theta} u) = \det(\rot u,u,|D|^{2\theta} u).
\]
For integer values of $2\theta$, one has:
\[
\det(\rot u,u,\rot_+^{2\theta} u) - \det(\rot u,u,\rot_-^{2\theta} u) = \det(\rot u,u,\rot^{2\theta} u).
\]
The determinants $\det(\rot u,u,\rot_\pm^{2\theta} u)$ are the geometric drive for the growth of the $\dot{H}^{\theta}$ norm of the spin-definite components of $u$. In particular, we have proven the following statement.
\begin{prop}\label{detD2th}
If $u$ is a smooth solution of Navier-Stokes, one has the following balance laws:
\begin{equation}
\norme[{\dot{H}^{\theta}}]{u(t)}^2
+ 2\nu \int_0^t \norme[{\dot{H}^{\theta+1}}]{u(t')}^2 dt'
+ 2\int_0^t\int_{\R^{3}}\det(\rot u,u,|D|^{2\theta} u) dx dt' = 
\norme[{\dot{H}^{\theta}}]{u_0}^2
\end{equation}
for any $\theta>0$, and the spin-definite variants~\eqref{eq:rot_theta};
when $\theta\in\mathbb{N}$, one can replace $|D|^{2\theta}$ by~$(-\Delta)^\theta$.
For any $n\in\mathbb{N}^\ast$, one has also
\begin{equation}
N^n_+(u,t)-N^n_-(u,t)
+ 2\int_0^t\int_{\R^{3}}\det(\rot u,u,\rot^n u) dx dt' = 
\norme[{\dot{H}^{n/2}}]{u_0^+}^2-\norme[{\dot{H}^{n/2}}]{u_0^-}^2
\end{equation}
where the definition~\eqref{222} is extended by
\begin{equation}
N^{n}_\pm(u,t) =
\norme[{\dot{H}^{n/2}}]{u^\pm (t)}^2
+ 2\nu \int_0^t \norme[{\dot{H}^{n/2+1}}]{u^\pm (t')}^2 dt' 
\end{equation}
and $u_0^\pm = \mathbb{Q}_\pm u_0$.
\end{prop}

The case $\theta=0$ is of special interest because, as for $\theta=1/2$, both critical determinants co\"incide.
\begin{prop}
If $u$ is a smooth solution of Navier-Stokes, one has the following balance laws:
\begin{equation}
\norme[L^2]{u^\pm(t)}^2 + 2\nu \int_0^t \norme[L^2]{\nabla u^\pm(t')}^2  dt'
\pm 2\int_0^t\int_{\R^{3}}\det(\rot u,u^-, u^+) dx dt' = 
\norme[L^2]{u^\pm_0}^2.
\end{equation}
In particular, the balance between the spin-definite components is ruled by:
\begin{equation}
N^0_+(u,t)-N^0_-(u,t) +4\int_0^t\int_{\R^{3}}\det(\rot u,u^-, u^+) dx dt'  = \norme[L^2]{u_0^+}^2-\norme[L^2]{u_0^-}^2
\end{equation}
\end{prop}
\proof
Using $u^\pm$ as a test function, one has
$\det(\rot u,u, u^\pm) = \det(\rot u,u^+ + u^-, u^\pm) = \pm \det(\rot u, u^-, u^+)$.
\cqfd

\subsection{Applications}

Regularity on $[0,T]\times\R^3$ is assured when the following inequality holds:
\begin{equation}
\exists\theta\geq 1/2, \qquad \int_0^T\int_{\R^{3}}\det(\rot u,u,|D|^{2\theta} u) dx dt \geq0.
\end{equation}
Of course, giving sense to the previous integral requires some a priori  knowledge that the solution is smooth.
However, if the inequality is satisfied for some $\theta\geq 1/2$ along a sequence of, \eg[,] Galerkine approximations
that converge to a given Leray solution $u$, then $u$ enjoys a uniform bound in $L^\infty([0,T];\dot{H}^\theta)$ and
therefore, according to~\cite{GKP2016}, is smooth on $[0,T]$.
To avoid making an assumption on approximating sequences, one can require instead the slightly stronger
property on a general Leray solution:
\begin{equation}
\exists\theta\geq 1/2, \quad \operatorname{a.e.} t\in[0,T] \qquad \int_{\R^{3}}\det(\rot u,u,|D|^{2\theta} u) dx \geq0
\end{equation}
with $u_0\in H^\theta$.
Then one can proceed as in the proof of Theorem~\ref{lem21} and show that the first time of singularity cannot occur before $T$.

\subsubsection{Uniqueness criterion based on critical determinants}
In this section, we revisit the weak-strong uniqueness result and investigate how the associated
stability estimate can be expressed in a more geometric way.
We refer the reader to \cite{G2006} and the references therein for an in-depth discussion of weak-strong uniqueness for Navier-Stokes.

\bigskip
Let us consider two Leray solutions $u_j$ ($j=1,2$) of the incompressible Navier-Stokes equation~\eqref{nsi++} and
their difference $\delta=u_1-u_2$. Using the energy inequality for each, one gets
\[
\norm{\delta(t)}_{L^{2}}^{2}+2\nu\int_{0}^{t}\norm{ \nabla \delta}_{L^{2}}^{2}
\leq \norm{u_1(0)}_{L^2}^{2}+\norm{u_2(0)}_{L^2}^{2}
-2\left(\poscal{u_1(t)}{u_2(t)}_{L^2}
+2\nu \int_0^t \poscal{\nabla u_1}{\nabla u_2}_{L^2}\right).
\]
The standard argument in favor of weak-strong uniqueness
consists in observing that each equation tested against (a regularized version of)
the other field ultimately gives:
\[
\poscal{u_1(t)}{u_2(t)}_{L^2}
+2\nu \int_0^t \poscal{\nabla u_1}{\nabla u_2}_{L^2}
=\poscal{u_1(0)}{u_2(0)}_{L^2}-\int_0^t \poscal{(\delta\cdot\nabla) u_1}{\delta}_{L^2},
\]
which implies
\begin{equation}
\norm{\delta(t)}_{L^{2}}^{2}+2\nu\int_{0}^{t}\norm{ \nabla \delta}_{L^{2}}^{2}
\leq \norm{\delta(0)}_{L^2}^{2} + 2\int_0^t \poscal{(\delta\cdot\nabla) u_1}{\delta}_{L^2}
\end{equation}
and, with Gronwall's inequality:
\begin{equation}
\norm{\delta(t)}_{L^{2}}^{2} \leq \norm{\delta(0)}_{L^{2}}^{2} \exp\left(
\int_0^t \norme[L^\infty]{\nabla u_1 (t')}  dt'\right).
\end{equation}
This control is enough to ensure the uniqueness of all Leray solutions stemming from $u_1(0)$ as long as $u_1$ remains smooth.
It remains nonetheless  quite crude. 

\medskip
Instead, using~\eqref{NL_weak_det}, let us rewrite the crucial step in a more geometric way:
\[
\poscal{u_1(t)}{u_2(t)}_{L^2}
+2\nu \int_0^t \poscal{\nabla u_1}{\nabla u_2}_{L^2}
+\int_0^t \det(\rot u_1, u_1, u_2) + \det(\rot u_2, u_2, u_1)
=\poscal{u_1(0)}{u_2(0)}_{L^2}.
\]
Observe that
\[
\det(\rot u_1, u_1, u_2) + \det(\rot u_2, u_2, u_1) =\det(\rot\delta, u_1, u_2)
= (u_1\times u_2) \cdot \rot\delta.
\]
As $\div\delta=0$, one has $\norme[L^2]{\rot\delta} = \norme[L^2]{\nabla\delta}$ and one can completely absorb the offending derivative:
\begin{equation}
\norm{\delta(t)}_{L^{2}}^{2}
\leq \norm{\delta(0)}_{L^2}^{2} 
+ \frac{1}{2\nu} \int_0^t \norme[L^2]{u_1\times u_2}^2.
\end{equation}
In particular, we have the following statement.
\begin{thm}
If $u_1$ and $u_2$ are two Leray solutions such that
\begin{equation}\label{04}
\norm{u_1\times u_2}_{L^2}^2\le \gamma(t) \norm{u_1-u_2}_{L^2}^2
\quad\text{with}\quad \gamma\in L^{1}([0,T])
\end{equation}
then for any $t\in[0,T]$, one has
\begin{equation}\label{geomStability}
\norm{\delta(t)}_{L^{2}}^{2} \leq \norm{\delta(0)}_{L^{2}}^{2} \exp\left(
\int_0^T \gamma(t') dt'\right).
\end{equation}
\end{thm}
For example, if $u_1\in L^2_t L^\infty_x$ we can apply this result because  $u_1\times u_2 =  - u_1 \times \delta$ and
 we recover a well-known case of weak-strong uniqueness. However,
the geometric assumption~\eqref{04} is a priori weaker if, for example, the two fields tend to line up when one of them grows unbounded.

\subsubsection{A variant of BKM based on critical determinants}
Formally, the standard argument for the Beale-Kato-Majda criterion \cite{MR763762}, \cite{MR1953068}
consists in writing the equation for vorticity~\eqref{eq_vorticity}
in weak form against $\omega$ itself, which gives:
\[
\norm{\omega(t)}_{L^{2}}^{2}
+2\nu\int_{0}^{t}
\norm{\rot\omega}_{L^{2}}^{2}=\norm{\omega_{0}}^{2}
+2\int_{0}^{t} \poscal{(\omega\cdot \nabla)u}{\omega}_{L^2}
\]
and thus, in particular
\[
\norm{\omega(t)}_{L^{2}}^{2}
\leq \norm{\omega_{0}}^{2}
+2\int_{0}^{t} \norme[L^2]{\omega}^2 \norme[L^\infty]{\omega}.
\]
Combined with Gronwall lemma, this ensures that the solution (of either Euler or Navier-Stokes) remains smooth as long as
\begin{equation}\label{bkm}
\int_{0}^{T}\norm{\omega(t)}_{L^{\infty}} dt<+\io.
\end{equation}
Let us present a variant of this computation, inspired by the previous critical determinants.

\medskip
Our starting point is similar, but we write the non-linear term slightly differently:
\begin{equation*}
\norm{\omega(t)}_{L^{2}}^{2}
+2\nu\int_{0}^{t} \norm{\rot\omega}_{L^{2}}^{2}
+2\int_{0}^{t}\poscal{\omega\times u}{\rot\omega}_{L^{2}}
=\norm{\omega_{0}}_{L^2}^{2}.
\end{equation*}
Now, if one splits $\nu=\nu_{1}+\nu_{2}$ with arbitrary values $\nu_{j}>0$, one gets:
\[
\norm{\omega(t)}_{L^{2}}^{2}
+2\nu_{1}\int_{0}^{t}\norm{\rot \omega+\frac{1}{2\nu_{1}} (\omega\times u)}_{L^{2}}^{2} +
2\nu_{2}\int_{0}^{t}
\norm{\rot\omega}_{L^{2}}^{2}
=\frac{1}{2\nu_{1}}
\int_{0}^{t}\norm{\omega\times u}^{2}_{L^{2}} +\norm{\omega_{0}}^{2}_{L^2}.
\]
In particular, one obtains \textbf{an estimate that is now specific to Navier-Stokes}:
\begin{equation}\label{critSpecificNS}
\norm{\omega(t)}_{L^{2}}^{2}+
2\nu_{2}\int_{0}^{t} \norm{\rot\omega}_{L^{2}}^{2}
\le \norm{\omega_{0}}^{2}_{L^2}+\frac{1}{2\nu_{1}}
\int_{0}^{t}\norm{\omega\times u}^{2}_{L^{2}}.
\end{equation}
Consequently, as  $\norm{\omega\times u}^{2}_{L^{2}}\le \norm{\omega}^{2}_{L^{2}}
\norm{u}^{2}_{L^{\io}}$, Gronwall's lemma ensures the regularity of the flow on $[0,T]$ provided that
\begin{equation}
\int_{0}^{T}\norm{u(t)}^{2}_{L^{\infty}} dt<+\io.
\end{equation}
This condition is the endpoint of the Prodi-Serrin $L^{q}_{t}L^{p}_{x}$ family with $\frac{2}{q}+\frac{3}{p}=1$.

\medskip
Let us finally point out that an interesting connection between the Beale-Kato-Majda criterion
and the theory of turbulence was established by Cheskidov \& Shvydkoy \cite{CS2010b}, who showed
that a condition
\begin{equation}\label{CS_BKM_variant}
\int_0^T \norme[B^0_{\infty,\infty}]{\omega_{\leq Q(t)}(t)} dt <\infty
\end{equation}
ensures the regularity of the flow on $[0,T]$. The dynamic wave-number $2^{Q(t)}$
separates high-frequency modes where viscosity prevails over the non-linear term
from the low-frequency modes where the Euler dynamics is dominant. It is defined by
\begin{equation}
Q(t) = \min\left\{ q\in\N \,;\, \forall p>q, \enspace 2^{-p} \norme[L^\infty]{\Delta_p u} < c_0\nu\right\}.
\end{equation}
The constant $c_0>0$ is absolute. The operators $\Delta_p$ are the Littlewood-Paley projection on the $p$-th dyadic shell
and $\omega_{\leq Q}$ denotes the corresponding projection on the spectral ball of radius $2^Q$.
Using this criterion and a relation between the time-average of $2^{Q(t)}$ and Kolmogorov's dissipation wave-number,
the authors of~\cite{CS2010b} provide a strong analytical support to the fact that 
\textbf{most turbulent flows} (\ie even mildly intermittent ones) \textbf{are actually regular solutions} of Navier-Stokes.

\medskip
In retrospect, this last observation makes the denomination of \textit{turbulent solution} given by Leray \cite{MR1555394} to his weak
solutions a now unnecessarily confusing linguistic choice and it may be unwise to propagate it in the modern literature: mathematical singularities, if
they exist, will be violent events that are likely to be of turbulent nature;
however, most turbulent flows of practical interest for engineering purposes are smooth,
albeit less smooth (\eg in terms of analyticity radius) than the laminar flows, and only display quasi-singularities.
Of course, this remark does not intend to denigrate in any way the admirable work of Jean Leray,
who was greatly ahead of his era and whose entire life \cite{M99}, \cite{C2007} 
was a tribute to what  a great mind can achieve in adversity,
when it is moved by an unquenchable curiosity and a strong sense of humanism.

\appendix
\section{Appendix}\label{appendix}

In this appendix, we recall some well known facts that bridge the standard vector calculus with
its geometric foundations. We denote by $\langle\cdot,\cdot\rangle$ the canonical Euclidian scalar product of $\R^3$
and by $(\vv{\mathbf e_1},\vv{\mathbf e_2},\vv{\mathbf e_3})$ the canonical orthonormal basis.
For a comprehensive introduction to geometrical hydrodynamics, we refer the reader to Arnold's works~\cite{MR1345386}, \cite{MR1612569}.

\subsection{Some vector calculus formulas}
Let us start with the defining identity for the vector product in $\R^{3}$.
\begin{claim}
 Let $A,B,C$ be vectors in  $\R^{3}$.
 Then we have 
 \begin{equation}
\poscal{A\times B}{C}_{\R^{3}}=\det(A,B,C)
\end{equation}
In particular if $\mathcal R$ is a  $3\times 3$ matrix,  we have 
\begin{equation}
{}^{t}{\mathcal R}\bigl({\mathcal R}A\times{\mathcal R}B)=(\det \mathcal R) (A\times B).
\end{equation}
\end{claim}
\begin{proof}Both sides are bilinear antisymmetric in $A,B$ thus one can reduce the identity to the sole
case~$A=\vv{\mathbf e_{1}}$ and~$B=\vv{\mathbf e_{2}}$, \ie
 $$
c_{3}=
\left\vert\begin{matrix}
 1&0&c_{1}\\0&1&c_{2}\\0&0&c_{3}
\end{matrix}\right\vert,
 $$
 which is obviously true.
\cqfd\end{proof}

\begin{claim}
Let $A,B,C,X,Y$ be vectors in $\R^{3}$. Then we have:
\begin{equation}
\det(A\times B, X, Y)=\poscal{A}{X}\poscal{B}{Y}-\poscal{B}{X}\poscal{A}{Y}
\end{equation}
and the triple cross-product formula:
 \begin{equation}\label{tvp}
(A\times B)\times C=\poscal{C}{A}B-\poscal{C}{B}A. 
\end{equation}
\end{claim}
\begin{proof}
For each identity, both sides are bilinear antisymmetric in $A,B$
The formulas reduce respectively to 
\[
 \left\vert
\begin{matrix}
 0&x_{1}&y_{1}\\0&x_{2}&y_{2}\\1&x_{3}&y_{3}
\end{matrix}
\right\vert
=x_{1}y_{2}-x_{2}y_{1}
\qquad\text{and}\qquad
\mattre{0}{0}{1}\times \mattre{c_{1}}{c_{2}}{c_{3}}= \mattre{-c_{2}}{c_{1}}{0},
\]
which are obviously true.
\cqfd\end{proof}
\begin{rmk}Equation~\eqref{tvp} implies the Jacobi identity
\begin{equation}\label{jacobi}
(A\times B)\times C+(B\times C)\times A+(C\times A)\times B=0,
\end{equation}
since the left-hand side of \eqref{jacobi} is also
$\underbrace{\langle{C},{A}\rangle} B
\aunderbrace[l1r]{-\langle{C},{B}\rangle A}
+\langle{A},{B}\rangle C\underbrace{-\langle{A},{C}\rangle B  }
+\aunderbrace[l1r]{\langle{B},{C}\rangle A}-\langle{B},{A}\rangle C  =0$. 
\end{rmk}

\subsection{Some differential calculus formulas}

An orientation of $\R^{3}$ is a choice of a non-trivial $\omega_{0}$ in the 3rd exterior power $\Lambda^3 \R^3$, \ie
a non-degenerate alternating trilinear form on $\R^3$.
\begin{dfn}
Let $w$ be a one-form in $\R^{3}$. We define the vector field $\curl w$ by the identity
 \begin{equation}\label{curl01}
 \iota_{(\curl w)} \omega_{0} = 
dw,
\end{equation}
where $\iota$ stands for the interior product.
\end{dfn}
\begin{rmk}
For $\omega_{0}=dx_{1}\wedge dx_{2}\wedge dx_{3}$ the interior product reads
\begin{equation}
\iota_X(\omega_0) = X_1 dx_2 \wedge dx_3
-X_2 dx_1 \wedge dx_3+X_3 dx_1 \wedge dx_2 
\end{equation}
and with $w=\sum w_{j} dx_{j}$ we recover the usual formula for the curl.
\end{rmk}
In particular, for a function $\alpha$, identifying a vector field $u$ to a one-form we find:
\begin{equation}
\curl (\alpha u)=\alpha\curl u+\nabla \alpha\times u.
\end{equation}
Next we investigate the $\curl$ of a general advection term and how these operators (do not) commute.
\begin{lemma}
 Let $u \in W^{1,p}_{\text{loc}}$ and  $v \in W^{2,p'}_{\text{loc}}$ be two vector fields on  $\R^{3}$
 for some $p\in [1,+\io]$. We have 
\begin{equation}\label{01}
\curl \bigl((u\cdot \nabla)v\bigr)=(u\cdot \nabla)\curl v
- ((\curl v)\cdot \nabla) u+(\dive u)(\curl v)+\sum_{1\le j\le 3}( \nabla u_{j}\times \nabla v_{j}).
\end{equation}
\end{lemma}
\begin{proof}
We use a geometric approach because any direct attempt leads to nightmarish computations.
We consider $u$ as a vector and $v$ as a 1-form and use Einstein summation convention freely:
\[
u=u_{j}\frac{\p}{\p x_{j}}
\qquad\text{and}\qquad
v=v_{j}dx_{j}.
\]
With  $\omega_{0}=dx_{1}\wedge dx_{2}\wedge dx_{3}$, recall that $  \iota_{(\curl v)} \omega_{0}=dv$
\ie $\curl v$  is a vector  and  $dv$ is a 2-form.
The Lie derivative $\mathcal L_{u}$ is defined by Elie Cartan's Formula:
\begin{equation}\label{}
\mathcal L_{u}(\omega) = \iota_u d\omega + d(\iota_u \omega).
\end{equation}
The convective term can be expressed as a 1-form in the following way:
\[
(u\cdot \nabla)v=\mathcal L_{u}(v_{j}) dx_{j}=\mathcal L_{u}(v)-v_{j}\mathcal L_{u}(dx_{j})
=
\mathcal L_{u}(v)-v_{j}d( \iota_u dx_j)=\mathcal L_{u}(v)-v_{j}du_{j}.
\]
As the Lie derivative commutes with exterior differentiation, one gets:
\[
d\bigl((u\cdot \nabla)v\bigr)=\mathcal L_{u}(dv)+du_{j}\wedge dv_{j}.
\]
Proceeding by identification, one gets
\begin{align*}
\iota_{\curl\left((u\cdot \nabla)v\right)}\omega_{0}
&=
d\bigl((u\cdot \nabla)v\bigr)
=\mathcal L_{u}(\iota_{(\curl v)}\omega_{0})+du_{j}\wedge dv_{j}
\\
&=\iota_{(\curl v)}\mathcal L_{u}(\omega_{0})
+\iota_{\mathcal L_{u}(\curl v)} \omega_{0}
+du_{j}\wedge dv_{j}
\\
&= (\text{div}u) \iota_{\curl v}\omega_0 + \iota_{[u,\curl v]} \omega_0 + du_{j}\wedge dv_{j},
\end{align*}
providing  \eqref{01} since 
$du_{j}\wedge dv_{j}= \iota_{(\nabla u_{j}\times \nabla v_{j})} \omega_{0}$.
\cqfd\end{proof}

\bigskip\noindent
The geometrical reason that  gives the convection term its cross-product structure (see identity~\eqref{curlInProd} when $u=v$) is the following.
\begin{lemma}
 Let $u, v$ be vector fields in $\R^{3}$. Then we have 
 \begin{equation}\label{appendix_curlucrossv}
(u\cdot \nabla)v+(v\cdot \nabla )u=\nabla(u\cdot v)-u\times \curl v-v\times \curl u.
\end{equation}
\end{lemma}
\begin{proof}
We introduce $\tilde u=u_{j}dx_{j}$, $\tilde v=v_{j}dx_{j}$ the two one-forms associated to $u$ and $v$
and proceed as in the proof of the previous Lemma:
\begin{align*}
(u\cdot\nabla) \tilde v+(v\cdot\nabla)\tilde u 
&=\mathcal L_{u}(v_{j})dx_{j}+\mathcal L_{v}(u_{j})dx_{j}\\
&=\mathcal L_{u}(\tilde v)-v_{j}\mathcal L_{u}(dx_{j})
+\mathcal L_{v}(\tilde u)-u_{j}\mathcal L_{v}(dx_{j})
\\
&= \iota_u  d\tilde v +\aunderbrace[l1r]{d(\iota_u  \tilde v)-v_{j}du_{j}}
+\iota_v  d\tilde u +d(\iota_v  \tilde u)
\aunderbrace[l1r]{-u_{j}dv_{j}}
\\
&=d(\iota_v  \tilde u)
+ \iota_u  \iota_{\curl v} \omega_0
+ \iota_v   \iota_{\curl u} \omega_0
\end{align*}
In the last expression, we used~\eqref{curl01} to expand $d\tilde u$ and $d\tilde v$.
The three underlined terms cancel each other out because $\iota_u  \tilde v = u_j v_j$.
Recall that the cross-product $u\times v$ is defined as a 1-form by the identity
\begin{equation}
(u\times v)\cdot w=
\omega_{0}(u,v,w)\quad \ie \quad u\times v=
\iota_v \iota_u \omega_{0}.
\end{equation}
We thus get 
\[
(u\cdot\nabla) \tilde v+(v\cdot\nabla)\tilde u=\nabla(u\cdot v)+\curl v\times u+\curl u\times v,
\]
which is the sought result.
\cqfd\end{proof}

\end{document}